\documentclass[11pt]{article}
\usepackage{enumerate}
\usepackage{lineno,hyperref,extarrows}
\usepackage{amsfonts}
\usepackage{amssymb}
\usepackage{mathrsfs}
\usepackage{srcltx}
\usepackage{verbatim}% 整段备注\begin{comment} \end{comment}
\usepackage{amsmath}
\usepackage{amsthm}
\usepackage{amstext}
\usepackage{amsopn}
\usepackage{graphicx}
\modulolinenumbers[4]
\allowdisplaybreaks[4]% 允许公式自动跨页

\newtheorem{theorem}{Theorem}[section]
\newtheorem{lemma}[theorem]{Lemma}

\newtheorem{corollary}[theorem]{Corollary}
\newtheorem{remark}[theorem]{Remark}

\theoremstyle{definition}
\newtheorem{definition}[theorem]{Definition}

\numberwithin{equation}{section}

\newcommand{\ba}{\begin{array}}
\newcommand{\ea}{\end{array}}
\newcommand{\f}{\frac}

\newcommand{\R}{\mathbb{R}}

\newcommand{\N}{\mathbb{N}}

\begin{document}
\date{}
 \title{\bf New approaches for Schr\"odinger equations with prescribed mass:
   The Sobolev  subcritical case
 and The Sobolev  critical case with mixed dispersion}
 %Normalized solutions Schr\"odinger equations with Sobolev  critical exponent and mixed dispersion
 %Improved results   for $L^2$-supercritical Schr\"odinger problems in the Sobolev  subcritical case and the Sobolev  critical case with mixed dispersion

\author{Sitong Chen,\ \  Xianhua Tang\footnote{The final revised version of this paper has been published in JDE.\ \ 
E-mail address: {\tt  mathsitongchen@mail.csu.edu.cn} (S.T. Chen),
{\tt tangxh@mail.csu.edu.cn} (X.H. Tang).}\\
{\small \it School of Mathematics and Statistics, HNP-LAMA, Central South University,}\\
{\small \it  Changsha, Hunan 410083, P.R.China}}
\maketitle
\begin{abstract}
In this paper, we prove
%develop several new critical point theories on a manifold to prove
the existence of normalized solutions for the following Schr\"odinger equation
\begin{equation*}
 \left\{
   \begin{array}{ll}
     -\Delta u-\lambda u=f(u), & x\in \R^N, \\
     \int_{\R^N}u^2\mathrm{d}x=c \\
   \end{array}
 \right.
 \end{equation*}
 with $N\ge3$, $c>0$, $\lambda\in \R$ and $f\in \mathcal{C}(\R,\R)$ in the Sobolev subcritical case
 with weaker $L^2$-supercritical conditions
 and in the Sobolev critical case when $f(u)=\mu |u|^{q-2}u+|u|^{2^*-2}u$ with $\mu>0$ and
 $2<q<2^*=\f{2N}{N-2}$  allowing to be $L^2$-subcritical, critical or supercritical.
 Our approach is based on several new critical point theories on a manifold, which not only help to
 weaken the previous $L^2$-supercritical conditions in the Sobolev subcritical case, but present
 an alternative scheme to construct bounded (PS) sequences on a manifold  when
 $f(u)=\mu |u|^{q-2}u+|u|^{2^*-2}u$ technically simpler than the Ghoussoub minimax principle \cite{Gh}
 involving topological arguments, as well as working for all $2<q<2^*$.
 In particular, we propose new strategies to control the energy level
 in the  Sobolev critical case which allow to treat, in a unified way, the dimensions $N=3$ and $N\ge 4$,
 and fulfill what were expected by Soave \cite{So-JFA} and by Jeanjean-Le \cite{JL-MA}.
  We believe that our approaches and strategies may be adapted and modified to
 attack more variational problems in the constraint contexts.

 \vskip2mm
 \noindent
 {\bf Keywords: }\ \ Nonlinear Schr\"odinger equation;  Normalized solution;   Sobolev critical growth;   Mixed nonliearities.

 \vskip2mm
 \noindent
 {\bf 2010 Mathematics Subject Classification.}\ \ 35J20, 35J62, 35Q55
\end{abstract}

%\journal{Journal of Differential Equations}

%%%%%%%%%%%%%%%%%%%%%%%
%% Elsevier bibliography styles
%%%%%%%%%%%%%%%%%%%%%%%
%% To change the style, put a % in front of the second line of the current style and
%% remove the % from the second line of the style you would like to use.
%%%%%%%%%%%%%%%%%%%%%%%

%% Numbered

%% Numbered without titles
%\bibliographystyle{model1a-num-names}

%% Harvard
%\bibliographystyle{model2-names.bst}\biboptions{authoryear}

%% Vancouver numbered
%\usepackage{numcompress}\bibliographystyle{model3-num-names}

%% Vancouver name/year
%\usepackage{numcompress}\bibliographystyle{model4-names}\biboptions{authoryear}

%% APA style
%\bibliographystyle{model5-names}\biboptions{authoryear}

%% AMA style
%\usepackage{numcompress}\bibliographystyle{model6-num-names}

%% `Elsevier LaTeX' style
%\bibliographystyle{elsarticle-num}
%%%%%%%%%%%%%%%%%%%%%%%

{\section{Introduction}}
 \setcounter{equation}{0}
 This paper is concerned with the nonlinear Schr\"odinger equation with an $L^2$-constraint
\begin{equation}\label{Pa}
 \left\{
   \begin{array}{ll}
     -\Delta u-\lambda u=f(u), & x\in \R^N, \\
     \int_{\R^N}u^2\mathrm{d}x=c, \\
   \end{array}
 \right.
 \end{equation}
 where $N\ge3$, $f\in \mathcal{C}(\R,\R)$, $c>0$ is a given mass, $\lambda\in \R$ appears as  a Lagrange multiplier which depends on the solution
  $u\in H^1(\R^N)$ and is not a priori given.

  The main feature of \eqref{Pa} is that the desired solutions have a priori prescribed $L^2$-norm,
 which are often referred to as normalized solutions in the literature,
that is, for given $c>0$,  a couple $(u,\lambda)\in H^1(\R^N)\times \R$ solves \eqref{Pa}.
From the physical viewpoint,  these solutions often offer a good insight of the dynamical properties of the stationary
solutions, such as the orbital stability or instability, see, for example, \cite{BJL,So-JDE,So-JFA}.
 This type of problem has attracted much attention in the community of nonlinear PDEs in the last decades.
 Under mild conditions on $f$, one can introduce the $\mathcal{C}^1$-functional
$\Phi: H^1(\R^N) \rightarrow \R$ defined by
 \begin{equation}\label{Ph}
  \Phi(u)=\f{1}{2}\|\nabla u\|_2^2-\int_{\R^N}F(u)\mathrm{d}x,
 \end{equation}
 where $F(t)=\int_0^tf(s)\mathrm{d}s$  for $t\in \R$.
It is standard that for prescribed $c>0$, a solution of \eqref{Pa} can be obtained as a critical point of the functional $\Phi$ constrained
to the sphere
\begin{equation}\label{Sc}
  \mathcal{S}_c:=\left\{u\in H^1(\R^N):\|u\|_2^2=c\right\}.
 \end{equation}
 As we know, the study of  \eqref{Pa} depends on the behavior of the nonlinearity $f$ at infinity,
 which gives rise to a new $L^2$-critical exponent $\bar{q}:=2+\f{4}{N}$,
 coming from the Gagliardo-Nirenberg inequality (see \cite[Theorem 1.3.7]{Caz}).
 One speaks of a $L^2$-subcritical case if $\Phi$ is bounded from below on $\mathcal{S}_c$ for any $c>0$,
 and of a $L^2$-supercritical case if $\Phi$ is unbounded from below on $\mathcal{S}_c$ for any $c>0$.
 One also refers to a $L^2$-critical case when the boundedness from below
 does depend on the value $c>0$.  We say that $u$ is a ground state solution to \eqref{Pa} if it is a solution
 having minimal energy among all the solutions which belong to $\mathcal{S}_c$.
 Compared with the $L^2$-subcritical case, more efforts are
 always needed in the study of the $L^2$-critical and $L^2$-supercritical cases.

 In this paper, we focus on not only normalized solutions of \eqref{Pa} in the $L^2$-critical and $L^2$-supercritical cases, but a more complicated situation when $f(u)=\mu |u|^{q-2}u+|u|^{2^*-2}u$ with $\mu>0$ and $2<q<2^*:=\f{2N}{N-2}$, namely the following equation with Sobolev  critical exponent and mixed dispersion:
 \begin{equation}\label{Pa1}
 \left\{
   \begin{array}{ll}
     -\Delta u+\lambda u=\mu |u|^{q-2}u+|u|^{2^*-2}u, & x\in \R^N, \\
     \int_{\R^N}u^2\mathrm{d}x=c,\\
   \end{array}
 \right.
 \end{equation}
 where, particularly, $q$ is allowed to be $L^2$-subcritical $2<q<2+\f{4}{N}$, $L^2$-critical $q=2+\f{4}{N}$, or $L^2$-supercritical $2+\f{4}{N}< q<2^*$.
 Let us describe the relevant works below that motivate our researches.

\medskip
{\subsection{$L^2$-supercritical problem \eqref{Pa}}}

  The first contribution to the $L^2$-supercritical case was made by Jeanjean \cite{Je-NA},
  where a radial solution of mountain pass type to \eqref{Pa} was found under the following conditions:
 \begin{itemize}
 \item[(H0)] $f$ is  odd;

 \item[(H1)] $f\in \mathcal{C}(\R, \R)$ and there exist $\alpha,\beta\in \R$ satisfying $2+\f{4}{N}<\alpha\le \beta<2^*=\f{2N}{N-2}$
  such that
  \begin{equation}\label{AR1}
    0<\alpha F(t)\le f(t)t\le \beta F(t), \ \ \ \ \forall\ t\in \R\setminus\{0\};
  \end{equation}
 \end{itemize}
 moreover a ground state solution was obtained if $f$ also satisfies
 \begin{itemize}
 \item[(H2)] the function $\tilde{F}(t):=f(t)t-2F(t)$ is of class $\mathcal{C}^1$ and
 $$
   \tilde{F}'(t)t>\left(2+\f{4}{N}\right)F(t), \ \ \ \ \forall \ t\in \R\setminus\{0\}.
 $$
 \end{itemize}
 Although the mountain pass geometry of $\Phi$ implies the existence of a Palais-Smale sequence $\{u_n\}$ (a (PS) sequence for short) at the mountain pass level, such a sequence may not be bounded. To overcome this  difficulty, Jeanjean constructed a good (PS) sequence having additional property related with  the Pohozaev type  identity, namely $\mathcal{P}(u_n)\to 0$,  where $\mathcal{P}: H^1(\R^2) \to \R$ is defined by
 \begin{equation}\label{Ju}
   \mathcal{P}(u):= \|\nabla u\|_2^2-\f{N}{2}\int_{\R^N}[f(u)u-2F(u)]\mathrm{d}x.
 \end{equation}
 %The inspiring part of the proofs is the neatly combination of the fiber map and  the mountain pass argument in the sense that
 The inspiring part of the proofs is the application of  the Ekeland principle to  the fibering map $ \tilde{\Phi}: H^1(\R^2)\times\R \to \R$ defined by
 \begin{align}\label{Tb0}
   \tilde{\Phi}(v,t) :=\f{e^{2t}}{2}\|\nabla v\|_2^2
              -\f{1}{e^{Nt}}\int_{\R^N}F\left(e^{Nt/2}v\right)\mathrm{d}x,
 \end{align}
 whose  mountain pass level on $\mathcal{S}_c\times \R$ equals to the one of $\Phi$ on $\mathcal{S}_c$. In particular, the first part of \eqref{AR1} in (H1) was used in a technical but essential way
 in showing not only the mountain pass geometry
 but the boundedness of the above (PS) sequence $\{u_n\}$ due to $\alpha\Phi(u_n)+o(1)=\alpha\Phi(u_n)- \mathcal{P}(u_n)$.
 It is reminiscent of the the classical Ambrosetti-Rabinowitz condition ((AR)-condition for short) introduced in
 \cite{AR-JFA} for the unconstrained superlinear problem:
 \begin{equation}\label{1}
  -\Delta u+ u=f(u)\ \  \hbox{in}\ \ \R^N.
 \end{equation}
 However, in contrast to unconstrained problems, the first part of \eqref{AR1}, as an analogue of (AR)-condition for \eqref{1}, is required in almost of the studies for the search of normalized solutions.
 It was not until 2022 that the first part of \eqref{AR1} was relaxed by Jeanjean-Lu \cite{JL-CVPDE} to the condition:
  \begin{itemize}
  \item [(H3)] $\lim_{t\to 0}f(t)/t^{1+\f{4}{N}}=0$ and $\lim_{t\to \infty}F(t)/t^{2+\f{4}{N}}=+\infty$,
 \end{itemize}
 in addition, provided the following monotonicity condition is satisfied:
  \begin{itemize}
  \item [(H4)]$[f(t)t-2F(t)]/|t|^{1+\f{4}{N}}t$ is nondecreasing on $(-\infty,0)$ and $(0,+\infty)$.
 \end{itemize}
 Particularly, ground state solutions were found by developing new robust arguments on the Pohozaev manifold defined by
 \begin{equation}\label{MP0}
   \mathcal{M}(c):=\left\{u\in \mathcal{S}_c: \mathcal{P}(u)=0\right\},
 \end{equation}
 combining the techniques  due to Szulkin and Weth \cite{SW-JFA,SW-Book}
 %the refined version of mini-max theorem
 with the mini-max approach in $\mathcal{M}(c)$ introduced by Ghoussoub \cite{Gh}.
 To our knowledge, there seems to be little progress in this direction  with the exception of \cite{JL-CVPDE}.
 Note that (H4) plays an analogous role as the Nehari-type condition for \eqref{1} for the search of ground state solutions. On the contrary, for the unconstrained problem \eqref{1}, (AR)-condition can be weakened to more general superlinear conditions without the Nehari-type condition.
 %Thus, one may wonder whether it is available in the $L^2$-supercritical constrained context, precisely,
 Thus, a natural question arises:
 {\bf{
\begin{itemize}
\item[(Q1)] Can we weaken {\rm(H1)} to more general $L^2$-supercritical conditions
  without imposing the monotonicity property on $f$ like {\rm(H4)}?
\end{itemize}
}}
\begin{comment}
\begin{center}
  \textit{ Can we weaken {\rm(H1)} to more general $L^2$-supercritical conditions
  without monotonicity property on $f$ like {\rm(H4)}?}
 \end{center}
 \end{comment}
 %To the best of our knowledge, this question  has not been solved yet.
 In the first part of this paper, we shall not only positively answer to the above question,
 but also provide a new approach to construct a bounded (PS) sequence of $\Phi\big|_{\mathcal{S}_c}$ at the mountain pass level.
 Before stating the result in this direction, let us introduce the following conditions:
 \begin{itemize}
 \item[(F0)] $f\in \mathcal{C}(\R,\R)$ and
  $$
    \lim_{t\to 0}\f{f(t)}{t}=0, \ \ \ \ \lim_{|t|\to +\infty}\f{|f(t)|}{|t|^{2^*-1}}<+\infty;
  $$
 \end{itemize}
 \begin{itemize}
 \item[(F1)] $f\in \mathcal{C}(\R,\R)$ and
  $$
    \lim_{t\to 0}\f{f(t)}{|t|^{1+\f{4}{N}}}=0, \ \ \ \ \lim_{|t|\to +\infty}\f{|f(t)|}{|t|^{2^*-1}}=0;
  $$
 \end{itemize}
 \begin{itemize}
 \item[(F2)] $0<\left(2+\f{4}{N}\right) F(t)\le f(t)t < \f{2N}{N-2} F(t)$ for all $t\in \R\setminus\{0\}$;

 \item[(F3)] there exists $\kappa>\f{N}{2}$ such that
  $$
    \limsup_{|t|\to+\infty}\f{[f(t)t-2F(t)]^{\kappa}}{t^{2\kappa}[Nf(t)t-(2N+4)F(t)]}<+\infty;
  $$

 \item[(F3$'$)] there exist $\kappa>\f{N}{2}$ and $\mathcal{C}_0>0$ such that
  $$
    \left(\f{f(t)t-2F(t)}{t^{2}}\right)^{\kappa}\le \mathcal{C}_0[Nf(t)t-(2N+4)F(t)],
      \ \ \ \ \forall \ t\in \R\setminus \{0\}.
  $$
 \end{itemize}
  Our idea of weakening (H1) is somehow inspired by Ding \cite{Di}
  where the unconstrained superlinear problem \eqref{1} was considered.
  Here comes the first result of this paper.

 \begin{theorem}\label{thm1.1} Let $c>0$.
 \begin{enumerate}[{\rm (i)}]
  \item If $f$ satisfies {\rm (F1)-(F3)}, then \eqref{Pa} admits a couple solution $(u_c,\lambda_c)
  \in H_{\mathrm{rad}}^1 (\R^N) \times (-\infty,0)$.

  \item If $f$ satisfies {\rm (F1), (F2)} and {\rm (F3$'$)}, then \eqref{Pa} admits a couple solution
  $(\bar{u}_c,\lambda_c) \in H_{\mathrm{rad}}^1 (\R^N) \times (-\infty,0)$ such that $\Phi(u_c)=\inf_{\mathcal{K}_c}\Phi$, where
 $$
   \mathcal{K}_c:=\left\{u\in \mathcal{S}_c\cap H_{\mathrm{rad}}^1 (\R^N): \ \Phi|_{\mathcal{S}_c}'(u)=0\right\}.
 $$
 \end{enumerate}
 \end{theorem}

 \begin{remark}\label{rem1.1}
  \begin{itemize}
  \item [{\rm(i)}] As pointed out by Jeanjean-Lu {\rm\cite[Remark 1.1 (i)]{JL-CVPDE}},
 there are no existence results on \eqref{Pa} so far without imposing related conditions as {\rm(H2)} or {\rm(H4)}.
 There are many functions $f(t)$ which satisfy {\rm(F1)-(F3)} but do neither {\rm(H2)} nor {\rm(H4)}. In this sense, Theorem {\rm\ref{thm1.1}} seems to be new, and improves and extends the related results on normalized solutions in the literature.

 \item [{\rm(ii)}] In the proof of Theorem {\rm\ref{thm1.1}}, inspired by {\rm\cite{BL2,Je-NA,WM}}, we develop new critical point theories on a manifold
 {\rm(}see Lemmas {\rm\ref{lem 2.3}} and {\rm\ref{lem 2.13}},
 Theorems {\rm\ref{thm 2.4}} and {\rm\ref{thm 2.14}},
 Corollaries {\rm\ref{cor 2.5}} and {\rm\ref{cor 2.15}{\rm)}},
 which help to generate a bounded {\rm(PS)}-sequence of $\Phi\big|_{\mathcal{S}_c}$ at the mountain pass level.
 They may be considered as counterparts of the deformation lemma and general minimax principle due to Willem {\rm\cite{WM}}  in the constraint context.
  We believe that these theories may be adapted and modified to attack more variational problems in the constraint contexts.
 \end{itemize}
 \end{remark}

 {\subsection{Mixed problem \eqref{Pa1} with Sobolev  critical exponent}}
 The second part is devoted to the study of normalized solutions for a more complex problem \eqref{Pa1} with Sobolev  critical exponent and mixed dispersion,
 which is the heart of this paper.
 The study of  such a problem is a very active topic nowadays,
 and can be as a counterpart of the Brezis-Nirenberg problem in the context of normalized solutions.
 It is well-known that solutions of \eqref{Pa1} are critical points of the functional
 \begin{equation}\label{Phu}
  \Phi_{\mu}(u)=\f{1}{2}\|\nabla u\|_2^{2}-\f{1}{2^*}\|u\|_{2^*}^{2^*}-\f{\mu}{q}\|u\|_q^q, \ \ \ \ \forall\ u\in H^1(\R^2)
 \end{equation}
 on the constraint $\mathcal{S}_c$, and $\Phi_{\mu}$ is unbounded from below on $\mathcal{S}_c$ due to $2^*=\f{2N}{N-2}>2+\f{N}{4}$.
 Similarly as \eqref{Ju} and \eqref{MP0}, we define the functional
 \begin{equation}\label{Pu}
   \mathcal{P}_{\mu}(u):= \|\nabla u\|_2^{2}-\|u\|_{2^*}^{2^*}-\mu\gamma_q\|u\|_q^q, \ \ \ \ \forall\ u\in H^1(\R^N)
 \end{equation}
 and the Pohozaev manifold
 \begin{equation}\label{Muc}
   \mathcal{M}_{\mu}(c):=\left\{u\in \mathcal{S}_c: \mathcal{P}_{\mu}(u)=0\right\}.
 \end{equation}
Compared with  the previous $L^2$-supercritical problem \eqref{Pa}, the study of \eqref{Pa1} is more delicate since we have to not only carefully analyse
 how a lower order term $|u|^{q-2}u$ affects the structure
 of  the constrained functional $\Phi_{\mu}\big|_{\mathcal{S}_c}$,
 but also solve the lack of compactness caused by Sobolev critical growth.
 This problem was firstly studied by Soave \cite{So-JFA}. Based on the fibration method of Pohozaev
 relying on the decomposition of Pohozaev manifold
  \begin{equation}\label{MD}
  \mathcal{M}_{\mu}(c)=\left\{u\in \mathcal{S}_c: \tilde{\phi}_u'(0)=0\right\}= \mathcal{M}_{\mu}^-(c)\cup \mathcal{M}_{\mu}^0(c)\cup \mathcal{M}_{\mu}^+(c),
 \end{equation}
% introduced in his previous paper  \cite{So-JDE} that the mixed nonlinearities were Sobolev subcritical,
  where $\tilde{\phi}_u(t)=\Phi\left(e^{Nt/2}u(e^tx)\right)$ for $ u\in H^1(\R^N)$ and $t\in \R$,
  \begin{equation}\label{M+-}
   \mathcal{M}_{\mu}^\pm(c):=\left\{u\in \mathcal{M}_{\mu}(c): \tilde{\phi}_u''(0)\gtrless 0\right\}, \ \
  \mathcal{M}_{\mu}^0(c):=\left\{u\in \mathcal{M}_{\mu}(c): \tilde{\phi}_u''(0)=0\right\}.
 \end{equation}

 Soave proved the following results:

 \smallskip
 \noindent
 {\bf Theorem [S] (\cite[Theorems 1.1]{So-JFA})}\ \ There exists a constant  $\alpha(N,q)>0$ depending on
 $N,q$, $\gamma_q:=\f{N(q-2)}{2q}$ and the best constant for the Gagliardo-Nirenberg inequality $\mathcal{C}_{N,q}$ (see \eqref{GN}) such that, if $\mu c^{\f{(1-\gamma_q)q}{2}}<\alpha(N,q)$, \eqref{Pa1} has a ground state solution $\tilde{u}$. Furthermore,
  \begin{itemize}
  \item [(i)] if $2<q<2+\f{4}{N}$, $\tilde{u}$ corresponds to a local minimizer satisfying $\inf_{ \mathcal{M}_{\mu}^+(c)}\Phi_{\mu}=\Phi_{\mu}(\tilde{u})<0$;

 \item [(ii)] if $2+\f{4}{N}\le q<2^*$, $\tilde{u}$ is stable and characterized as a solution of mountain-pass type with $0<\inf_{ \mathcal{M}_{\mu}^-(c)}\Phi_{\mu}=\Phi_{\mu}(\tilde{u})<\f{1}{N}\mathcal{S}^{\f{N}{2}}$, where
     and in the sequel, $\mathcal{S}$ denotes the best constant for the Sobolev inequality (see \eqref{Sob}).
% where $\alpha(N,q)=+\infty$ for $N=3,4$, and $\alpha(N,q)$ is finite for $N\ge 5$;
 \end{itemize}
 Subsequently, by introducing a set $  V(c):=\{u\in \mathcal{S}_c: \|\nabla u\|_2^2<\rho_0\}$ having the property that
 \begin{equation}\label{mmu1}
   m_{\mu}(c):=\inf_{u\in V(c)}\Phi_{\mu}(u)<0<\inf_{u\in \partial V(c)}\Phi_{\mu}(u)\ \ \hbox{with}\ \
   \partial V(c):=\{u\in \mathcal{S}_c: \|\nabla u\|_2^2=\rho_0\},
 \end{equation}
 where $\rho_0$ and $c_0$ are given as follows:
 \begin{align}\label{rhod}
   \rho_0:=\left[\f{2^*\mu \alpha_0\mathcal{C}_{N,q}^q\mathcal{S}^{2^*/2}}{q\alpha_2}\right]
          ^{\f{2}{\alpha_2+\alpha_0}}c^{\f{\alpha_1}{\alpha_0+\alpha_2}}
 \end{align}
 and
 \begin{align}\label{c0d}
   c_0:=  \left[\f{2^*\alpha_0{\mathcal{S}}^{2^*/2}}{\alpha_0+\alpha_2}
           \left(\f{q\alpha_2}{2^*\mu\alpha_0\mathcal{C}_{N,q}^q\mathcal{S}^{2^*/2}}\right)
           ^{\f{\alpha_2}{\alpha_0+\alpha_2}}\right]^{\f{N}{2}},
 \end{align}
 with
 \begin{equation}\label{al}
   \alpha_0:=2-\f{N(q-2)}{2}, \ \ \alpha_1:=\f{2N-q(N-2)}{2}, \ \ \alpha_2:=\f{4}{N-2}.
 \end{equation}
 Jeanjean-Jendrej-Le-Visciglia \cite{JJLV} proved that for any $c\in (0,c_0)$, the set of ground state solutions is orbitally stable for
 the case $2<q<2+\f{4}{N}$.
 %which settled a open question raised by Soave \cite{So-JFA}.
 Note that such a structure of local minima for the case $2<q<2+\f{4}{N}$,
 suggests the possibility to search for a solution lying at a mountain pass level,
 which was proposed by Soave \cite[Remark 1.1]{So-JFA} as a conjecture.
 Recently, this conjecture has been confirmed by Jeanjean-Le \cite{JL-MA} and Wei-Wu \cite{WW-JFA}.
 Let us now recall the results obtained there.
 Based on the same decomposition \eqref{MD} as that in \cite{So-JFA}, with some new energy estimates, Wei-Wu \cite{WW-JFA}
 complemented the results of the above Theorem [S] from three respects:
   \begin{itemize}
  \item in the case $2<q<2+\f{4}{N}$, a second solution $u_{\mu,c}^{-}\in \mathcal{M}_{\mu}^{-}(c)$ with
 $\Phi_{\mu}(u_{\mu,c}^{-})=\inf_{\mathcal{M}_{\mu}^-(c)}\Phi_{\mu}$
 was found if $\mu c^{\f{(1-\gamma_q)q}{2}}<\alpha(N,q)$ which satisfies $0<\Phi_{\mu}(u_{\mu,c}^{-})<\inf_{ \mathcal{M}_{\mu}^-(c)}\Phi_{\mu}+\f{1}{N}\mathcal{S}^{\f{N}{2}}$;
 \item  in the case $q=2+\f{4}{N}$, it was shown that \eqref{Pa1}  has no ground state solutions for
 $\mu c^{\f{(1-\gamma_q)q}{2}}\ge\alpha(N,q)$;
 \item  in the case $2+\f{4}{N}< q<2^*$, the existence range that $\mu c^{\f{(1-\gamma_q)q}{2}}<\alpha(N,q)$ was extended to all $c>0$.
 \end{itemize}
 %Based the existence result of ground state solutions obtained in \cite{JJLV}, Jeanjean-Le
 Instead of the fibration method of Pohozaev used in \cite{So-JFA,WW-JFA},
 %by a better understanding of the geometry of $\Phi_{\mu}\big|_{\mathcal{S}_c}$,
 \begin{comment}
 \begin{equation*}
  \left\{u\in \mathcal{S}_c: \Phi_{\mu}(u)= m_{\mu}(c):=\inf_{u\in V(c)}\Phi_{\mu}(u)\right\},
 \end{equation*}

  a natural question arises:
   {\bf{
\begin{itemize}
\item[(Q$_2$)] Can we develop a general minimax principle on some good manifolds, not involving technical topological arguments,
 as a counterpart of the one of Willem \cite[Theorem 2.8]{WM}?
\end{itemize}
}}
 One may still wonder
\begin{itemize}
\item[(Ex)] {\it Thus, it is natural to expect that there exists a general minimax principle
 on the  manifold works for all $2<q<2^*$, not involving technical topological arguments,
 as a counterpart of the one of Willem} \cite[Theorem 2.8]{WM}.
\end{itemize}
 \end{comment}
 by introducing a new set of mountain pass level, directly connected with the
decomposition
 \begin{equation*}
  \mathcal{M}_{\mu}(c)= \widehat{\mathcal{M}}_{\mu}^-(c)\cup \widehat{\mathcal{M}}_{\mu}^+(c) \ \
 \hbox{ with} \ \  \widehat{\mathcal{M}}_{\mu}^\pm(c):=\left\{u\in \mathcal{M}_{\mu}(c): \Phi_{\mu}(u)\gtrless 0\right\},
 \end{equation*}
 and studying its relation with $V(c)$, Jeanjean-Le   \cite{JL-MA}
 proved that for any $c\in (0,c_0)$ ($c_0$ is given in \eqref{c0d}) and $N\ge 4$, there exists a second solution $v_c\in \mathcal{S}_c$ of mountain pass type satisfying $0<\Phi_{\mu}(v_c)<m_{\mu}(c)+\f{1}{N}\mathcal{S}^{\f{N}{2}}$, which is not a ground state solution, where $V(c)$ and $m_{\mu}(c)$ are given by \eqref{mmu1}.

 As turned out in the aforementioned papers,  two ingredients in the search of a solution of mountain pass type for \eqref{Pa1} are essential:
 I) obtaining a (PS) sequence $\{u_n\}$ at the mountain pass level having additional property $\mathcal{P}(u_n)
 \rightarrow 0$; II) proving the compactness of the obtained (PS) sequence $\{u_n\}$. Let us remark them in detail below.

 To do item I), all of them adopted the Ghoussoub minimax principle \cite{Gh} on the manifold.
 % which is a very powerful tool in the study of unconstrained variational problems.
 This strategy is very effective for $2<q<2+\f{4}{N}$ as well as  $2+\f{4}{N}\le q<2^*$,
 but involves technical topological arguments based on
 $\sigma$-homotopy stable family of compact subsets of $\mathcal{M}_{\mu}(c)$, moreover, the arguments are very complicated.  Note that the minimax approach developed by Jeanjean \cite{Je-NA} works only for $2+\f{4}{N}\le q<2^*$, even though  it is technically simpler.
 % As we all know, in order to apply the Willem's principle, one only needs to analyze geometric structures of the corresponding functional,
 % other sophisticated techniques are not required.as a counterpart of the one of Willem \cite[Theorem 2.8]{WM}
 Thus, one may ask: {\bf{
\begin{itemize}
\item[(Q2)] If is it possible to establish a general minimax principle
 on the  manifold not involving topological arguments, technically simpler than \cite{Gh} and working for all $2<q<2^*$?
\end{itemize}
}}
To do item II), the crucial point is to derive a
 strict upper bound of mountain pass level:
 \begin{equation}\label{**}
  M_{\mu}(c)< \left\{
   \begin{array}{ll}
     m_{\mu}(c)+\f{1}{N}\mathcal{S}^{\f{N}{2}},  & \ \hbox{if}\ 2<q<2+\f{4}{N}, \\
    \f{1}{N}\mathcal{S}^{\f{N}{2}}, & \ \hbox{if}\ 2+\f{4}{N}\le q<2^*\\
   \end{array}
 \right.
 \end{equation}
 through the use of testing functions, as firstly pointed out in \cite{JL-MA} for $2<q<2+\f{4}{N}$ and in \cite{So-JFA} for $2+\f{4}{N}\le q<2^*$.
 This strict inequality can help to guarantee that the obtained (PS) sequence does not carry a bubble which,
 by vanishing when passing to the weak limit, would prevent its strong convergence, like the classical Brezis-Nirenberg problem.
 In form, such a threshold of compactness, is an analogue of the unconstrained Sobolev critical problem \eqref{1}
 with concave-convex nonlinearities. Even if this idea to get \eqref{**} may somehow been generated,
  its achievement  is rather involved since the choice of test functions is more delicate in the $L^2$-constrained context.
 Indeed, to do that, for $2<q<2+\f{4}{N}$ and $\mu c^{\f{(1-\gamma_q)q}{2}}\ge\alpha(N,q)$, Wei-Wu \cite{WW-JFA}
 used the radial superposition of  a local minima $u_c^+$ and a suitable family of truncated extremal Sobolev functions
 located in a set where the local minima solution takes its greater
 values. The strategy in \cite{WW-JFA}, recording of the one introduced
 by Tarantello \cite{Tar}, is that the interaction
 decreases the mountain pass value of  $\Phi_{\mu}\big|_{\mathcal{S}_c}$ with respect to the case where the two supports would be disjoint,
 moreover,  it works for all dimensions $N\ge 3$.
 Instead of  $\mu c^{\f{(1-\gamma_q)q}{2}}<\alpha(N,q)$ proposed by Soave \cite{So-JFA},
 Jeanjean-Le \cite{JL-MA} considered the range $c\in (0,c_0)$, and constructed non-radial test functions
 which could be viewed as the sum of a truncated extremal Sobolev function centered at the origin and
 of a local minima $u_c^+$  translated far away  from the origin. Unlike \cite{WW-JFA},
 their construction aims at separating sufficiently the regions where the functions concentrate and to show
 that the remaining interaction can be assumed sufficiently.
 However, only when $N\ge 4$ does this strategy work, and it was pointed out in \cite[Remark 1.17]{JL-MA} that:
  {\it It is not clear to us if this limitation is due to the approach we have developed or if the case $N=3$ is
fundamentally distinct from the case $N\ge 4$. We believe it would be interesting to investigate in that direction.}
Thus, a natural question arises:
{\bf{
\begin{itemize}
\item[(Q3)] Can we construct alternative testing functions, working for all dimensions $N\ge 3$,
 to derive that $ M_{\mu}(c)<m_{\mu}(c)+\f{1}{N}\mathcal{S}^{\f{N}{2}}$ for $2<q<2+\f{4}{N}$ and $c\in (0,c_0)$?
\end{itemize}
}}
\noindent
Note that the energy estimate for the case $2+\f{4}{N}\le q<2^*$ is quite different.
To do that, Soave \cite{So-JFA} used truncated and normalized extremal Sobolev functions centered at the origin as test functions, and the
dilations preserving the $L^2$-norm as test paths.
Although this choice appears to be natural and more easier in comparison to the case $2<q<2+\f{4}{N}$,  to get the desired upper estimates,
not only the cases $2+\f{4}{N}< q<2^*$ and $q=2+\f{4}{N}$, but the dimensions $N=3$ and $N\ge 4$ in each case, all needed to be treated separately in the proofs.
Later, the same strategy was employed by  Wei-Wu \cite{WW-JFA}.
For this case, we also refer to \cite[Theorem 1.1]{AJM}, in which $\mu$ must be large enough and the mountain pass level was controlled small enough
by the approximate scheme $M_{\mu}(c)\to 0$ as $\mu\to +\infty$. Now, pursuing the study of  \cite{So-JFA,WW-JFA},
 in the case $2+\f{4}{N}\le q<2^*$, it is natural to ask:
 {\bf{
\begin{itemize}
\item[(Q4)] Can we find an unified scheme to treat cases $2+\f{4}{N}< q<2^*$ and $q=2+\f{4}{N}$, as well as the dimensions $N=3$ and $N\ge 4$ in each case?
\end{itemize}
}}

 In the second part of this paper, we are interested in Questions (Q2)-(Q4), and we shall solve them in turn by developing some new analytical strategies and techniques.

 \par
  To present the abstract minimax principle on the manifold involving Question (Q2), inspired by
 \cite{BL2,Je-NA,WM}, we develop new critical point theories on a manifold, see Lemmas \ref{lem 2.3} and
 \ref{lem 2.13}, Theorems \ref{thm 2.4} and \ref{thm 2.14}, Corollaries \ref{cor 2.5} and \ref{cor 2.15},
 which present a different approach to construct a bounded (PS) sequence on a manifold technically simpler than \cite{Gh}. They may be considered as counterparts of the deformation lemma and general minimax principle due to Willem
 \cite{WM} in the constraint context. We believe, these results can be adapted to more $L^2$-constrained problems. In the proofs, we use the general deformation lemma on a manifold (see Lemma \ref{lem 2.3}), instead of  technical topological arguments of \cite{Gh}. This shows that critical point theories introduced in the first part,
 mentioned in Remark \ref{rem1.1} (ii), are of future development and applicability to some extend.

 \par
  On Questions (Q3) and (Q4), we have the following results, respectively, in which
  $\Phi_{\mu}$, $\mathcal{M}_{\mu}(c)$, $m_{\mu}(c)$ and $c_0$ are defined before by \eqref{Phu}, \eqref{Muc}, \eqref{mmu1} and \eqref{c0d}.

 \begin{theorem}\label{thm 1.2} Let $N\ge 3, 2<q<2+\f{4}{N}, \mu>0$ and $c\in (0, c_0)$. Then \eqref{Pa1}
 has a second couple solution $(u_c,\lambda_c) \in H_{\mathrm{rad}}^1 (\R^N) \times (-\infty,0)$ such that
 \begin{equation}\label{M24}
   0<\Phi_{\mu}(u_c)< m_{\mu}(c)+\f{1}{N}\mathcal{S}^{\f{N}{2}}.
 \end{equation}
 \end{theorem}

 \begin{theorem}\label{thm 1.3} Let $N\ge 3$, $c>0$ and $2+\f{4}{N}\le q<2^*$. Then, \eqref{Pa1}
 has a couple solution $(\bar{u}_c,\lambda_c) \in H^1 (\R^N) \times (-\infty,0)$ such that
 \begin{equation}\label{gs1}
   \Phi_{\mu}(\bar{u}_{c})=\inf_{\mathcal{M}_{\mu}(c)}\Phi_{\mu}
 \end{equation}
  \begin{enumerate}[{\rm (i)}]
  \item for any $\mu>0$ if $2+\f{4}{N}<q<2^*$;
  \item for any $0<\mu\le \f{1}{2\gamma_{\bar{q}}c^{2/N}\mathcal{C}_{N,\bar{q}}^{\bar{q}}}$ if $q=\bar{q}=2+\f{4}{N}$.
 \end{enumerate}
 \end{theorem}

 \begin{remark}\label{rem1.2}
 Theorem {\rm\ref{thm 1.2}} gives a second solution of mountain pass type by constructing a more simple geometry of
 the mountain pass than the one in {\rm\cite{JL-MA}}, and Theorem  {\rm\ref{thm 1.3}} complements the
 results of  {\rm\cite[Theorem 1.1]{WW-JFA}}.
  In particular, to derive  the strict inequality  {\rm\eqref{**}}, we introduce alternative choices of test functions in the cases
   $2<q<2+\f{4}{N}$ and $2+\f{4}{N}\le q<2^*$,
  which, together with subtle estimates and analyses, allows us to treat uniformly the dimensions $N\ge 3$,
  see Lemmas {\rm\ref{lem 4.4}} and {\rm\ref{lem 4.14}}.
  Thus, Questions {\rm(Q3)} and {\rm(Q4)} are fully settled.
  We believe that our strategy of energy estimates is helpful to other $L^2$-constrained problems with Sobolev critical growth.
 \end{remark}

 \bigskip
  \par
 The paper is organized as follows. Section 2, is devoted to establish several new critical point theories on a manifold
 and finish the proof of Theorem \ref{thm 2.5},
 which will be applied to prove Theorems \ref{thm1.1}, \ref{thm 1.2} and \ref{thm 1.3}.
 In Section 3, we study $L^2$-supercritical problem \eqref{Pa} and prove Theorem \ref{thm1.1}.
In Section 4, we study Normalized solutions for mixed problem \eqref{Pa1} with Sobolev  critical exponent,
 and finish the proofs of Theorems \ref{thm 1.2} and  \ref{thm 1.3}.

 \bigskip
 \par
 Throughout the paper, we make use of the following notations:

 \par
  $\bullet$  $H_{\mathrm{rad}}^1(\R^2):=\{u\in H^1(\R^2)\ \big|\ u(x)=u(|x|)\ \hbox{a.e. in } \R^N\}$;

     $\bullet$ \ $L^s(\R^2) (1\le s< \infty)$  denotes the Lebesgue space with the norm $\|u\|_s
 =\left(\int_{\R^2}|u|^s\mathrm{d}x\right)^{1/s}$;

 \par
  $\bullet$ \ For any $N \ge 3$ there exists
 an optimal constant $\mathcal{S} > 0$ depending only on $N$, such that
 \begin{equation}\label{Sob}
   \mathcal{S}\|u\|_{2^*}^{2}\le \|\nabla u\|_2^{2},\ \ \forall \ u\in \mathcal{D}^{1,2}(\R^N)
    \ \ \mbox{(Sobolev inequality)}
 \end{equation}
 and $\mathcal{C}_{N,s}>0$ such that
 \begin{equation}\label{GN}
   \|u\|_s\le \mathcal{C}_{N,s}\|\nabla u\|_2^{\gamma_s}\|u\|_{2}^{1-\gamma_s}\ \ \ \ \forall \ u\in H^{1}(\R^N)
    \ \ \mbox{(Gagliardo-Nirenberg inequality)};
 \end{equation}

 \par
     $\bullet$ \ For any $x\in \R^N$ and $r>0$, $B_r(x):=\{y\in \R^N: |y-x|<r \}$
     and $B_r=B_r(0)$;

 \par
     $\bullet$ \ $C_1, C_2,\cdots$ denote positive constants possibly different in different places.

{\section{Critical point theories on a manifold}}
 \setcounter{equation}{0}
 \par

 In this section, we give several new critical point theorems on a manifold, and complete the proof of Theorem \ref{thm 2.5}.
  We shall apply these theorems to obtain (PS) sequences restricted on a manifold and study problems \eqref{Pa} and \eqref{Pa1}.

 \par
  Let $H$ be a real Hilbert space whose norm and scalar product will be denoted
  respectively by $\|\cdot\|_H$ and $(\cdot, \cdot)_H$. Let $E$ be a real Banach space with norm $\|\cdot\|_E$. We
 assume throughout this section that
 \begin{equation}\label{EH}
   E\hookrightarrow H \hookrightarrow E^*
 \end{equation}
 with continuous injections, where $E^*$ is the dual space of $E$. Thus $H$ is identified with its dual space.
 We will always assume in the sequel that $E$ and $H$ are infinite dimensional spaces. We consider the manifold
 \begin{equation}\label{MD}
   M:=\{u\in E: \|u\|_H=1\}.
 \end{equation}
 $M$ is the trace of the unit sphere of $H$ in $E$ and is, in general, unbounded. Throughout the paper, $M$
 will be endowed with the topology inherited from $E$. Moreover $M$ is a submanifold of $E$ of codimension
 1 and its tangent space at a given point $u\in M$ can be considered as a closed subspace of $E$ of codimension 1, namely
 \begin{equation}\label{TM}
   T_uM:=\{v\in E: (u,v)_H=0\}.
 \end{equation}

 \par
  We consider a functional $\varphi: E\rightarrow \R$ which is of class $\mathcal{C}^1$ on $E$.
 We denote by $\varphi|_{M}$ the trace of $\varphi$ on $M$. Then $\varphi|_{M}$  is a $\mathcal{C}^1$
 functional on $M$, and for any $u\in M$,
 \begin{equation}\label{DM}
   \langle\varphi|_{M}'(u), v\rangle=\langle\varphi'(u), v\rangle,\ \ \forall \ v\in T_uM.
 \end{equation}
 In the sequel, for any $u\in M$, we define the norm $\left\|\varphi|_{M}'(u)\right\|$ by
 \begin{equation}\label{DMN}
   \left\|\varphi|_{M}'(u)\right\|=\sup_{v\in T_uM,\|v\|_E=1}|\langle\varphi'(u), v\rangle|.
 \end{equation}

 Before this, we present some basic definitions and results.
 For any $\delta>0$ and any set $A\subset E$, we denote $A_{\delta}:=\{v\in E: \|u-v\|<\delta,
  \ \forall \ u\in A\}$. For any $a\in \R$ and any $\varphi: E\rightarrow \R$, we denote
 $\varphi^a:=\{u\in E : \varphi(u)\le a\}$.

 \begin{definition}\label{def 2.0} {\rm\cite{BL2,WM}} A pseudo-gradient vector for $\varphi|_{M}$ at $u\in
 \tilde{M}:=\{u\in M: \varphi|_{M}'(u)\ne 0\}$ is a vector $v\in T_uM$ such that
 $$
   \|v\|_E\le 2\left\|\varphi|_{M}'(u)\right\|,\ \ \langle\varphi|_{M}'(u),v\rangle\ge \left\|\varphi|_{M}'(u)\right\|^2.
 $$
 Put $T(M):=\bigcup_{u\in M}(T_uM$). A mapping (lifting) $g: \tilde{M}\rightarrow T(M)$ is called a pseudo-gradient vector field for $\varphi|_{M}$ on $\tilde{M}$ if $g$ is locally Lipschitz continuous, $g(u)\in T_uM$ for $u\in \tilde{M}$, and $g(u)$ is a pseudo-gradient vector for $\varphi|_{M}$ at any $u\in \tilde{M}$.
 \end{definition}

 \begin{lemma}\label{lem 2.1} {\rm\cite{BL2,WM}} If $\varphi|_{M}\in \mathcal{C}^1(M,\R)$, then there exists a pseudo-gradient vector field.
 \end{lemma}

 \begin{lemma}\label{lem 2.2} {\rm\cite{BL2}} Let $\{u_n\}\subset M$ be a bounded sequence in $E$. Then
 the following are equivalent:
 \begin{enumerate}[{\rm (i)}]
  \item $\|\varphi|_{M}'(u_n)\|\rightarrow 0$ as $n\rightarrow \infty$;
  \item $\varphi'(u_n)-\langle\varphi'(u_n),u_n\rangle u_n\rightarrow 0$ in $E^*$ as $n\rightarrow \infty$.
 \end{enumerate}
 \end{lemma}

 The following two statements may be considered as counterparts of the deformation lemma and general minimax principle due to Willem {\rm\cite{WM}} in the constraint context.

 \begin{lemma}\label{lem 2.3} Let $\varphi\in \mathcal{C}^1(E,\R)$, $S\subset M$, $a\in \R$,
 $\varepsilon,\delta>0$ such that
 \begin{equation}\label{D12}
   u\in M\cap \varphi^{-1}([a-2\varepsilon,a+2\varepsilon])\cap S_{2\delta}\Rightarrow
   \left\|\varphi|_{M}'(u)\right\|\ge \f{8\varepsilon}{\delta}.
 \end{equation}
 Then, there exists $\eta\in \mathcal{C}([0,1]\times M, M)$ such that
 \begin{enumerate}[{\rm (i)}]
  \item $\eta(t,u)=u$, if $t=0$, or if $u\notin M\cap \varphi^{-1}([a-2\varepsilon,a+2\varepsilon])\cap S_{2\delta}$;
  \item $\eta\left(1,\varphi^{a+\varepsilon}\cap S\right)\subset \varphi^{a-\varepsilon}$;
  \item for every $t\in [0,1]$, $\eta(t,\cdot): M\rightarrow M$ is a homeomorphism;
  \item $\|\eta(t,u)-u\|_E\le \delta,\ \forall \ u\in M,\ t\in [0,1]$;
  \item for every $u\in M$, $\varphi(\eta(t,u))$ is non-increasing on $t\in [0,1]$;
  \item $\varphi(\eta(t,u))<a,\ \forall \ u\in M\cap \varphi^a\cap S_{\delta},\ t\in [0,1]$.
 \end{enumerate}
 \end{lemma}

 \begin{proof}
 By Lemma \ref{lem 2.1}, there exists a pseudo-gradient vector field $g(u)$ for $\varphi|_{M}$ on
 $\tilde{M}=\{u\in M: \varphi|'_{M}(u)\ne 0\}$, i.e. for every $u\in \tilde{M}$, there exists $g(u)\in T_uM$
 such that
 \begin{equation}\label{D15}
   \|g(u)\|_E\le 2\left\|\varphi|_{M}'(u)\right\|,\ \ \left\langle\varphi|_{M}'(u),g(u)\right\rangle
    \ge \left\|\varphi|_{M}'(u)\right\|^2.
 \end{equation}
 Let
 $$
  A:=\{u\in M: a-2\varepsilon\le \varphi(u)\le a+2\varepsilon\}\cap S_{2\delta},
 $$
 $$
  B:=\{u\in M: a-\varepsilon\le \varphi(u)\le a+\varepsilon\}\cap S_{\delta}
 $$
 and
 \begin{equation}\label{D16}
   \varrho(u):=\f{\mathrm{dist}(u,M\setminus A)}{\mathrm{dist}(u,M\setminus A)+\mathrm{dist}(u,B)}.
 \end{equation}
 Define a vector field $W$ on $M$ by
 \begin{equation}\label{Wu}
   W(u):=\begin{cases}
    -\varrho(u)\f{g(u)}{\|g(u)\|^2_E}, & \mbox{if}\ u\in A,\\
    0, & \mbox{if}\ u\in M\setminus A
  \end{cases}
 \end{equation}
 Thus, $W$ is a locally Lipschitz continuous vector field on $M$ and $W(u)\in T_uM$ for $u\in \tilde{M}$. By \eqref{D12} and \eqref{D15}, one has
 \begin{equation}\label{D17}
   \|W(u)\|_E\le \f{\delta}{8\varepsilon},\ \ \forall \ u\in M.
 \end{equation}

 \par
  As in \cite{BL2}, we let $h: \R^{+}\rightarrow \R^{+}$ be a Lipschitz continuous function such that
 $h(1)=0$, $h(s)=0$ if $s\in [0,\f{1}{2}]\cup [2,+\infty)$ and $0\le h(s)\le 1$ for $s\in [\f{1}{2},2]$.
 Define a vector field $\tilde{W}$ on $E$ by
 \begin{equation}\label{tWu}
   \tilde{W}(u):=\begin{cases}
    h(\|u\|_H)W\left(\f{u}{\|u\|_H}\right), & \mbox{if}\ \f{1}{2}\le \|u\|_H\le 2,\\
    0, & \mbox{otherwise}.
  \end{cases}
 \end{equation}
 Then $\tilde{W}(u)$ is a locally Lipschitz continuous vector field on $E$ and satisfies
 $(u,\tilde{W}(u))_H=0$ for $u\in E$.

  Now consider the Cauchy problem in $E$:
 \begin{equation}\label{fl}
  \begin{cases}
   \f{\mathrm{d}}{\mathrm{d}t}\zeta(t,u)=\tilde{W}(\zeta(t,u)),\\
    \zeta(0,u)=u.
  \end{cases}
 \end{equation}
 By a standard arguments, for all $u\in E$, equation \eqref{fl} has a unique solution $\zeta(t,u)\in E$, defined for all $t\in \R$.  It follows from \eqref{D16}, \eqref{Wu}, \eqref{tWu} and \eqref{fl} that
 \begin{equation}\label{D20}
   \f{\mathrm{d}}{\mathrm{d}t}\|\zeta(t,u)\|_H^2=2\left(\zeta(t,u),\f{\mathrm{d}}{\mathrm{d}t}\zeta(t,u)\right)_H
    =2(\zeta(t,u),\tilde{W}(\zeta(t,u)))_H=0,\ \ \forall \ u\in E, \ t\in \R.
 \end{equation}
 This shows that $\|\zeta(t,u)\|_H=\|u\|_H=1$ for $u\in M$ and $t\in \R$. Thus
 \begin{equation}\label{D22}
   \zeta(t,u)\in M,\ \ \forall \ u\in M, \ t\in \R.
 \end{equation}
By \eqref{D15}, \eqref{D16}, \eqref{D17}, \eqref{fl} and \eqref{D22}, we have
 \begin{eqnarray}\label{D24}
   \|\zeta(t,u)-u\|_E
   &  =  & \left\|\int_{0}^{t}W(\zeta(s,u))\mathrm{d}s\right\|_E\nonumber\\
   & \le & \int_{0}^{t}\|W(\zeta(s,u))\|_E\mathrm{d}s\le \f{\delta t}{8\varepsilon},
           \ \ \forall \ u\in M, \ t\ge 0
 \end{eqnarray}
 and
 \begin{eqnarray}\label{D26}
   \f{\mathrm{d}}{\mathrm{d}t}\varphi(\zeta(t,u))
   &  =  & \left\langle\varphi'(\zeta(t,u)),\f{\mathrm{d}}{\mathrm{d}t}\zeta(t,u)\right\rangle\nonumber\\
   &  =  & \langle\varphi'(\zeta(t,u)),W(\zeta(t,u))\rangle\nonumber\\
   &  =  & -\f{\varrho(\zeta(t,u))}{\|g(\zeta(t,u))\|_E^2}\langle\varphi'(\zeta(t,u)),g(\zeta(t,u))\rangle\nonumber\\
   & \le & -\f{1}{4}\varrho(\zeta(t,u)),\ \ \forall \ u\in M, \ t> 0.
 \end{eqnarray}
 Set $\eta(t,u):=\zeta(8\varepsilon t,u)$. From the continuous dependence on $u$ in \eqref{fl} and \eqref{D22},
 it follows that, for all $t\in [0,1]$, $\eta(t,\cdot)$ is a homeomorphism: $M\rightarrow M$, i.e. (iii) holds.
 By \eqref{D16}, \eqref{Wu}, \eqref{D22}, \eqref{D24} and \eqref{D26}, it is easy to verify that
 (i), (iv)-(vi) hold.

 \par
   Finally, we prove (ii) holds also. Let $u\in \varphi^{a+\varepsilon}\cap S$. If there is
 $t_0\in [0,8\varepsilon]$ such that $\varphi(\zeta(t_0,u))<a-\varepsilon$, then it follows from \eqref{D26} that $\varphi(\zeta(8\varepsilon,u))<a-\varepsilon$ and (ii) is satisfied. If
 \begin{equation}\label{D30}
   a-\varepsilon\le \varphi(\zeta(t,u))\le a+\varepsilon,\ \ \forall \ t\in [0,8\varepsilon],
 \end{equation}
 then \eqref{D22} and \eqref{D24} implies that $\zeta(t,u)\in M\cap S_{\delta}$ for $t\in [0,8\varepsilon]$, and so
 $\zeta(t,u)\in B$ for $t\in [0,8\varepsilon]$, it follows from  \eqref{D26}  and \eqref{D30} that
 \begin{eqnarray*}\label{D32}
   \varphi(\zeta(8\varepsilon,u))
   &  =  & \varphi(u)+\int_{0}^{8\varepsilon}\f{\mathrm{d}}{\mathrm{d}t}\varphi(\zeta(t,u))\mathrm{d}t\nonumber\\
   & \le & \varphi(u)-\f{1}{4}\int_{0}^{8\varepsilon}\varrho(\zeta(t,u))\mathrm{d}t\nonumber\\
   & \le & a+\varepsilon-2\varepsilon=a-\varepsilon,
 \end{eqnarray*}
 and (ii) is also satisfied.

 \end{proof}

 \begin{theorem}\label{thm 2.4}
  Assume that $\theta\in \R$, $\varphi\in \mathcal{C}^1(E,\R)$ and $\Upsilon\subset M$ is a closed set. Let
 \begin{equation}\label{Ga0}
   \Gamma:=\left\{\gamma\in \mathcal{C}([0,1], M): \gamma(0)\in \Upsilon, \varphi(\gamma(1))
    <\theta\right\}.
 \end{equation}
 If $\varphi$ satisfies
 \begin{equation}\label{D40}
   a:=\inf_{\gamma\in \Gamma}\max_{t\in [0, 1]}\varphi(\gamma(t))
     >b:=\sup_{\gamma\in \Gamma}\max\left\{\varphi(\gamma(0)), \varphi(\gamma(1))\right\},
 \end{equation}
 then, for every $\varepsilon\in (0,(a-b)/2)$, $\delta>0$ and $\hat{\gamma}\in \Gamma$ such that
 \begin{equation}\label{D42}
   \sup_{t\in [0, 1]}\varphi(\hat{\gamma}(t))\le a+\varepsilon,
 \end{equation}
 there exists $u\in M$ such that
 \begin{enumerate}[{\rm (i)}]
  \item $a-2\varepsilon\le \varphi(u)\le a+2\varepsilon$;
  \item $\min_{t\in [0,1]}\|u-\hat{\gamma}(t)\|_E\le 2\delta$;
  \item $\left\|\varphi|_{M}'(u)\right\|\le \f{8\varepsilon}{\delta}$.
 \end{enumerate}
 \end{theorem}

 \begin{proof}
 Since $\varepsilon\in (0,(a-b)/2)$, so we have
 \begin{equation}\label{D45}
   a-2\varepsilon>b.
 \end{equation}
 Suppose the conclusions are false. We apply Lemma \ref{lem 2.3} with $S=\hat{\gamma}([0,1])$. Then, there exists
 $\eta\in \mathcal{C}([0,1]\times M, M)$ such that (i)-(vi) in Lemma \ref{lem 2.3} hold.
 We define $\tilde{\gamma}(t)=\eta(1,\hat{\gamma}(t))$. Then it follows from \eqref{D40}, \eqref{D45} and
 (i) in Lemma \ref{lem 2.3} that
 $$
   \tilde{\gamma}(0)=\eta(1,\hat{\gamma}(0))=\hat{\gamma}(0),\ \ \tilde{\gamma}(1)=\eta(1,\hat{\gamma}(1))
   =\hat{\gamma}(1),
 $$
 so $\tilde{\gamma}\in \Gamma$. By \eqref{D42} and (ii) in Lemma \ref{lem 2.3}, one has
 $$
   \sup_{t\in [0,1]}\varphi(\tilde{\gamma}(t))=\sup_{t\in [0,1]}\varphi(\eta(1,\hat{\gamma}(t)))
     \le a-\varepsilon,
 $$
 contradicting the definition of $a$.
 \end{proof}

 As a consequence of Theorem \ref{thm 2.4}, we have:

 \begin{corollary}\label{cor 2.5}
  Assume that $\theta\in \R$, $\Upsilon\subset M$ and $\varphi\in \mathcal{C}^1(E,\R)$ satisfy \eqref{Ga0} and \eqref{D40}. Then there exists a sequence $\{u_n\}\subset M$ satisfying
 \begin{equation}\label{PS}
   \varphi(u_n)\rightarrow a,\ \ \left\|\varphi|_M'(u_n)\right\|\rightarrow 0.
 \end{equation}
 \end{corollary}

 \begin{theorem}\label{thm 2.5}
 Let $\varphi\in \mathcal{C}^1(E,\R)$ and $A\subset E$. If there exists $\rho>0$ such that
 \begin{equation}\label{D50}
   a:=\inf_{v\in M\cap A}\varphi(v)
     <b:=\inf_{v\in M\cap (A_{\rho}\setminus A)}\varphi(v),
 \end{equation}
 then, for every $\varepsilon\in (0,(b-a)/2)$, $\delta\in (0,\rho/2)$ and $w\in M\cap A$ such that
 \begin{equation}\label{D52}
   \varphi(w)\le a+\varepsilon,
 \end{equation}
 there exists $u\in M$ such that
 \begin{enumerate}[{\rm (i)}]
  \item $a-2\varepsilon\le \varphi(u)\le a+2\varepsilon$;
  \item $\|u-w\|_E\le 2\delta$;
  \item $\left\|\varphi|_{M}'(u)\right\|\le 8\varepsilon/\delta$.
 \end{enumerate}
 \end{theorem}

 \begin{proof}
 Since $\varepsilon\in (0,(b-a)/2)$, which implies
 \begin{equation}\label{D55}
   a+2\varepsilon<b.
 \end{equation}
 Suppose the thesis is false. We apply Lemma \ref{lem 2.3} with $S=\{w\}$,
 there exists $\eta\in \mathcal{C}([0,1]\times M, M)$ such that the conclusions (i)-(vi) in Lemma
 \ref{lem 2.3} hold. Set $\tilde{v}=\eta(1,w)$. Then (iv) in Lemma \ref{lem 2.3} implies that $\tilde{v}\in S_{\delta}\subset A_{\delta}$.
 If $\tilde{v}\in A_{\delta}\setminus A$, then it follows from (i), (v) in Lemma \ref{lem 2.3} and
 \eqref{D55} that
 $$
   a+\varepsilon\ge \varphi(w)=\varphi(\eta(0,w))\ge \varphi(\eta(1,w))=\varphi(\tilde{v})\ge b>a+2\varepsilon,
 $$
 which is absurd. If $\tilde{v}\in A$, then both \eqref{D52} and (ii) in Lemma \ref{lem 2.3} yield
 $$
   \varphi(\tilde{v})=\varphi(\eta(1,w))\le a-\varepsilon,
 $$
 contradicting the definition of $a$.
 \end{proof}

 As direct corollaries of Theorem \ref{thm 2.5}, we have:
 \begin{corollary}\label{cor 2.6}
  Let $\varphi\in \mathcal{C}^1(E,\R)$ and $A\subset E$ and let $\{w_n\}\subset M\cap A$ be a minimizing
 sequence for $\inf_{M\cap A}\varphi$. If there exists $\rho>0$ such that \eqref{D50} holds,
 then there exists a sequence $\{u_n\}\subset M$ satisfying
 \begin{equation}\label{PS}
   \varphi(u_n)\rightarrow a, \ \ \|u_n-w_n\|_E\rightarrow 0, \ \ \left\|\varphi|_M'(u_n)\right\|\rightarrow 0.
 \end{equation}
 \end{corollary}

 \begin{corollary}\label{cor 2.20}
 Let $\varphi\in \mathcal{C}^1(E,\R)$ and $A\subset E$. If there exist $\rho>0$ and $\bar{u}\in M\cap A$ such that
 \begin{equation}\label{D50+}
   \varphi(\bar{u})=\inf_{v\in M\cap A}\varphi(v)
     <\inf_{v\in M\cap (A_{\rho}\setminus A)}\varphi(v),
 \end{equation}
 then $\varphi|_{M}'(\bar{u})=0$.
 \end{corollary}

 \begin{proof}
 Let $a:=\inf_{v\in M\cap A}\varphi(v)$ and $b:=\inf_{v\in M\cap (A_{\rho}\setminus A)}\varphi(v)$. Suppose the above conclusion is false,
 then there exists $\varepsilon_0\in (0,(b-a)/2)$ such that $\|\varphi|_{M}'(\bar{u})\|\ge \varepsilon_0$.
 Since $\varphi\in \mathcal{C}^1(E,\R)$, then there exists $\delta_0\in (0,\rho/2)$ such that
 \begin{equation}\label{D60}
   u\in M, \ \ \|u-\bar{u}\|_E<\delta_0 \ \Rightarrow \ \|\varphi|_{M}'(u)\|\ge \frac{\varepsilon_0}{2}.
 \end{equation}
 By Theorem \ref{thm 2.5}, for $\varepsilon=\frac{\varepsilon_0}{n}$, $\delta=\frac{\delta_0}{4}$ and $\bar{u}\in M\cap A$ such that
 \begin{equation*}\label{D62}
   \varphi(\bar{u})\le a+\frac{\varepsilon_0}{n},
 \end{equation*}
 there exists $u\in M$ such that for all $n\in \N$,
 \par
  i) $a-\frac{2\varepsilon_0}{n}\le \varphi(u)\le a+\frac{2\varepsilon_0}{n}$; \ \ ii) $\|u-\bar{u}\|_E\le \frac{\delta_0}{2}$; \ \
  iii) $\left\|\varphi|_{M}'(u)\right\|\le \frac{32\varepsilon_0}{n\delta_0}$.
 \par\noindent
 These contradict with \eqref{D60} for large $n\in \N$.

 \end{proof}

\begin{remark}\label{rem2.1} As one will see in the next section, although a {\rm(PS)} sequence $\{u_n\}$ of $\Phi\big|_{\mathcal{S}_c}$
can be generated by applying Corollary {\rm\ref{cor 2.5}}, it not sufficient to derive the boundedness of $\{u_n\}$ in $H^1(\R^N)$.
In the spirit of Jeanjean {\rm\cite{Je-NA}}, to do that, we hope to find  a {\rm(PS)}  sequence having additional property $\mathcal{P}(u_n)\to 0$.
Different from that of {\rm\cite{Je-NA}}, we will employ critical point theorems for the $\mathcal{C}^1$-functional
 $\tilde{\Phi}(u,t)=\Phi(e^{Nt/2}u(e^tx))$ for $(u,t)\in H^1(\R^N)\times \R$, which will be given in the following.
 \end{remark}

 \par
  Let $E\times\R$ be equipped with the scalar product
 $$
   ((u,\tau),(v,\sigma))_{E\times\R}:=(u,v)_{H}+\tau\sigma,\ \ \forall \ (u,\tau), (v,\sigma)\in E\times \R,
 $$
 and corresponding norm
 $$
   \|(u,\tau)\|_{E\times\R}:=\sqrt{\|u\|^2_{H}+\tau^2},\ \ \forall \ (u,\tau)\in E\times \R.
 $$
  Next, we consider a functional $\tilde{\varphi}: E\times\R\rightarrow \R$ which is of class $\mathcal{C}^1$
 on $E\times\R$. We denote by $\tilde{\varphi}|_{M\times\R}$ the trace of $\tilde{\varphi}$ on $M\times\R$.
 Then $\tilde{\varphi}|_{M\times\R}$  is a $\mathcal{C}^1$ functional on $M\times\R$, and for any $(u,\tau)
 \in M\times\R$,
 \begin{equation}\label{DMM}
   \langle\tilde{\varphi}|_{M\times\R}'(u,\tau), (v,\sigma)\rangle
     :=\langle\tilde{\varphi}'(u,\tau), (v,\sigma)\rangle,\ \ \forall \ (v,\sigma)\in \tilde{T}_{(u,\tau)}(M\times\R),
 \end{equation}
 where
 \begin{equation}\label{TMM}
   \tilde{T}_{(u,\tau)}(M\times\R):=\{(v,\sigma)\in E\times\R: (u,v)_H=0\}.
 \end{equation}
 In the sequel, for any $(u,\tau)\in M\times\R$, we define the norm $\left\|\tilde{\varphi}|_{M\times\R}'(u,\tau)\right\|$ by
 \begin{equation}\label{DMN}
   \left\|\tilde{\varphi}|_{M\times\R}'(u,\tau)\right\|=\sup_{(v,\sigma)\in \tilde{T}_{(u,\tau)}(M\times\R),\|(v,\sigma)\|_{E\times\R}=1}|\langle\tilde{\varphi}'(u,\tau),
    (v,\sigma)\rangle|.
 \end{equation}

 \begin{definition}\label{def 2.10} {\rm\cite{BL2}} A pseudo-gradient vector for $\tilde{\varphi}|_{M\times\R}$ at $(u,\tau)\in \widetilde{M\times\R}:=\{(u,\tau)\in M\times\R: \tilde{\varphi}|_{M\times\R}'(u,\tau)\ne 0\}$ is
 a vector $(v,\sigma)\in \tilde{T}_{(u,\tau)}(M\times\R)$ such that
 $$
   \|(v,\sigma)\|_{E\times\R}\le 2\left\|\tilde{\varphi}|_{M\times\R}'(u,\tau)\right\|,\ \ \langle\tilde{\varphi}|_{M\times\R}'(u,\tau),(v,\sigma)\rangle
    \ge \left\|\tilde{\varphi}|_{M\times\R}'(u,\tau)\right\|^2.
 $$
 Put
 $$
  \tilde{T}(M\times\R):=\bigcup_{(u,\tau)\in M\times\R}(\tilde{T}_{(u,\tau)}(M\times\R)).
 $$
 A mapping (lifting) $g: \widetilde{M\times\R}\rightarrow \tilde{T}(M\times\R)$ is called a pseudo-gradient vector field for $\tilde{\varphi}|_{M\times\R}$ on $\widetilde{M\times\R}$ if $g$ is locally Lipschitz continuous,
 $g(u,\tau)\in \tilde{T}_{(u,\tau)}(M\times\R)$ for $(u,\tau)\in \widetilde{M\times\R}$, and $g(u,\tau)$ is a pseudo-gradient vector for $\tilde{\varphi}|_{M\times\R}$ at any $(u,\tau)\in \widetilde{M\times\R}$.
 \end{definition}

 \begin{lemma}\label{lem 2.11} {\rm\cite{BL2}} If $\tilde{\varphi}|_{M\times\R}\in \mathcal{C}^1(M\times\R,\R)$,
 then there exists a pseudo-gradient vector field.
 \end{lemma}

 \begin{lemma}\label{lem 2.13} Let $\tilde{\varphi}\in \mathcal{C}^1(E\times\R,\R)$, $S\subset M\times\R$,
 $a\in \R$, $\varepsilon,\delta>0$ such that
 \begin{equation}\label{F12}
   (u,\tau)\in (M\times\R)\cap \tilde{\varphi}^{-1}([a-2\varepsilon,a+2\varepsilon])\cap S_{2\delta}\Rightarrow
   \left\|\tilde{\varphi}|_{M\times\R}'(u,\tau)\right\|\ge \f{8\varepsilon}{\delta}.
 \end{equation}
 Then, there exists $\eta\in \mathcal{C}([0,1]\times (M\times\R), M\times\R)$ such that
 \begin{enumerate}[{\rm (i)}]
  \item $\eta(t,(u,\tau))=(u,\tau)$, if $t=0$, or if $(u,\tau)\notin (M\times\R)\cap \tilde{\varphi}^{-1}([a-2\varepsilon,a+2\varepsilon])
  \cap S_{2\delta}$;
  \item $\eta\left(1,\tilde{\varphi}^{a+\varepsilon}\cap S\right)\subset \tilde{\varphi}^{a-\varepsilon}$;
  \item for every $t\in [0,1]$, $\eta(t,\cdot): M\times\R\rightarrow M\times\R$ is a homeomorphism;
  \item $\|\eta(t,(u,\tau))-(u,\tau)\|_{E\times\R}\le \delta,\ \forall \ (u,\tau)\in M\times\R,\ t\in [0,1]$;
  \item for every $(u,\tau)\in M\times\R$, $\tilde{\varphi}(\eta(t,(u,\tau)))$ is non-increasing on $t\in [0,1]$;
  \item $\tilde{\varphi}(\eta(t,(u,\tau)))<a,\ \forall \ (u,\tau)\in (M\times\R)\cap \tilde{\varphi}^a\cap S_{\delta},\ t\in [0,1]$.
 \end{enumerate}
 \end{lemma}

 \begin{proof}
 By Lemma \ref{lem 2.11}, there exists a pseudo-gradient vector field
 $g(u,\tau)\in \tilde{T}_{(u,\tau)}(M\times\R)$ for $\tilde{\varphi}|_{M\times\R}$ on
 $\widetilde{M\times\R}=\{(u,\tau)\in M\times\R: \tilde{\varphi}|'_{M\times\R}(u,\tau)\ne 0\}$, i.e. for every
 $(u,\tau)\in \widetilde{M\times\R}$, there exists $g(u,\tau):=(g_1(u,\tau),g_2(u,\tau))$
 $\in\tilde{T}_{(u,\tau)}(M\times\R)$ such that
 \begin{equation}\label{F14}
   (u,g_1(u,\tau))_H=0
 \end{equation}
 and
 \begin{equation}\label{F15}
   \|g(u,\tau)\|_{E\times\R}\le 2\left\|\tilde{\varphi}|_{M\times\R}'(u,\tau)\right\|,\ \ \left\langle\tilde{\varphi}|_{M\times\R}'(u,\tau),g(u,\tau)\right\rangle
    \ge \left\|\tilde{\varphi}|_{M\times\R}'(u,\tau)\right\|^2.
 \end{equation}
 Let
 $$
  A:=\{(u,\tau)\in M\times\R: a-2\varepsilon\le \tilde{\varphi}(u,\tau)\le a+2\varepsilon\}\cap S_{2\delta},
 $$
 $$
  B:=\{(u,\tau)\in M\times\R: a-\varepsilon\le \tilde{\varphi}(u,\tau)\le a+\varepsilon\}\cap S_{\delta}
 $$
 and
 \begin{equation}\label{F16}
   \varrho(u,\tau):=\f{\mathrm{dist}((u,\tau),(M\times\R)\setminus A)}{\mathrm{dist}((u,\tau),(M\times\R)\setminus A)+\mathrm{dist}((u,\tau),B)}.
 \end{equation}
 Define a vector field $W$ on $M\times\R$ by
 \begin{equation}\label{Wuu}
   W(u,\tau):=\begin{cases}
    -\varrho(u,\tau)\f{g(u,\tau)}{\|g(u,\tau)\|^2_{E\times\R}}, & \mbox{if}\ (u,\tau)\in A,\\
    (0,0), & \mbox{if}\ (u,\tau)\in (M\times\R)\setminus A.
  \end{cases}
 \end{equation}
 Thus, $W$ is a locally Lipschitz continuous vector field on $M\times\R$ and $W(u,\tau)\in \tilde{T}_{(u,\tau)}(M\times\R)$ for $(u,\tau)\in \widetilde{M\times\R}$. By \eqref{F12} and \eqref{F15}, one has
 \begin{equation}\label{F17}
   \|W(u,\tau)\|_{E\times\R}\le \f{\delta}{8\varepsilon},\ \ \forall \ (u,\tau)\in M\times\R.
 \end{equation}
 \par
  Now let $h: \R^{+}\rightarrow \R^{+}$ be a Lipschitz continuous function such that
 $h(1)=0$, $h(s)=0$ if $s\in [0,\f{1}{2}]\cup [2,+\infty)$ and $0\le h(s)\le 1$ for $s\in [\f{1}{2},2]$.
 Define a vector field $\tilde{W}$ on $E$ by
 \begin{equation}\label{tWuu}
   \tilde{W}(u,\tau):=\begin{cases}
    h(\|u\|_H)W\left(\f{u}{\|u\|_H},\tau\right), & \mbox{if}\ \f{1}{2}\le \|u\|_H\le 2, \ \tau\in \R\\
    0, & \mbox{otherwise}.
  \end{cases}
 \end{equation}
 Then $\tilde{W}(u,\tau)=(\tilde{W}_1(u,\tau),\tilde{W}_2(u,\tau))$ is a locally Lipschitz continuous vector field on $E\times \R$ and satisfies $(u,\tilde{W}_1(u,\tau))_H=0$ for $u\in E$.

 \par
  Consider the Cauchy problem in $E\times\R$:
 \begin{equation}\label{fll}
  \begin{cases}
   \f{\mathrm{d}}{\mathrm{d}t}\zeta(t,(u,\tau))=\tilde{W}(\zeta(t,(u,\tau))),\\
    \zeta(0,(u,\tau))=(u,\tau).
  \end{cases}
 \end{equation}
 By a standard arguments, for all $(u,\tau)\in E\times\R$, equation \eqref{fll} has a unique solution
 $\zeta(t,(u,\tau)):=(\zeta_1(t,(u,\tau)),\zeta_2(t,(u,\tau)))\in E\times\R$, defined for all $t\in \R$.
 It follows from \eqref{F14}, \eqref{F16}, \eqref{Wuu}, \eqref{tWuu} and \eqref{fll} that
 \begin{eqnarray}\label{F20}
   \f{\mathrm{d}}{\mathrm{d}t}\|\zeta_1(t,(u,\tau))\|_H^2
     &  =  & 2\left(\zeta_1(t,(u,\tau)),\f{\mathrm{d}}{\mathrm{d}t}\zeta_1(t,(u,\tau))\right)_H\nonumber\\
     &  =  & 2(\zeta_1(t,(u,\tau)),\tilde{W}_1(\zeta(t,(u,\tau))))_H\nonumber\\
     &  =  & 0,\ \ \forall \ u\in E, \ \tau, t\in \R.
 \end{eqnarray}
 This shows that $\|\zeta_1(t,(u,\tau))\|_H=\|u\|_H=1$ for $u\in M$ and $\tau, t\in \R$. Thus
 \begin{equation}\label{F22}
   \zeta(t,(u,\tau))\in M\times\R,\ \ \forall \ (u,\tau)\in M\times\R, \ t\in \R.
 \end{equation}
 From \eqref{F15}, \eqref{F16}, \eqref{Wuu}, \eqref{F17}, \eqref{tWuu}, \eqref{fll} and \eqref{F22}, we have
 \begin{eqnarray}\label{F24}
   \|\zeta(t,(u,\tau))-(u,\tau)\|_{E\times\R}
   &  =  & \left\|\int_{0}^{t}W(\zeta(s,(u,\tau)))\mathrm{d}s\right\|_{E\times\R}\nonumber\\
   & \le & \int_{0}^{t}\|W(\zeta(s,(u,\tau)))\|_{E\times\R}\mathrm{d}s\le \f{\delta t}{8\varepsilon},\nonumber\\
   &     &  \ \ \ \ \forall \ (u,\tau)\in M\times\R, \ t\ge 0
 \end{eqnarray}
 and
 \begin{eqnarray}\label{F26}
   \f{\mathrm{d}}{\mathrm{d}t}\tilde{\varphi}(\zeta(t,(u,\tau)))
   &  =  & \left\langle\tilde{\varphi}'(\zeta(t,(u,\tau))),\f{\mathrm{d}}
            {\mathrm{d}t}\zeta(t,(u,\tau))\right\rangle\nonumber\\
   &  =  & \langle\tilde{\varphi}'(\zeta(t,(u,\tau))),W(\zeta(t,(u,\tau)))\rangle\nonumber\\
   &  =  & -\f{\varrho(\zeta(t,(u,\tau)))}{\|g(\zeta(t,(u,\tau)))\|_{E\times\R}^2}
            \langle\tilde{\varphi}'(\zeta(t,(u,\tau))),g(\zeta(t,(u,\tau)))\rangle\nonumber\\
   & \le & -\f{1}{4}\varrho(\zeta(t,(u,\tau))),\ \ \forall \ (u,\tau)\in M\times\R, \ t> 0.
 \end{eqnarray}
 Set $\eta(t,(u,\tau)):=\zeta(8\varepsilon t,(u,\tau))$. From the continuous dependence on $(u,\tau)$ in
 \eqref{fll} and \eqref{F22}, it follows that, for all $t\in [0,1]$, $\eta(t,\cdot)$ is a homeomorphism: $M\times\R\rightarrow M\times\R$, i.e. (iii) holds. By \eqref{F16}, \eqref{Wuu}, \eqref{F24} and  \eqref{F26}, it is easy to verify that
 (i), (iv)-(vi) hold.

 \par
   Finally, we prove (ii) holds also. Let $(u,\tau)\in \tilde{\varphi}^{a+\varepsilon}\cap S$. If there is
 $t_0\in [0,8\varepsilon]$ such that $\tilde{\varphi}(\zeta(t_0,(u,\tau)))<a-\varepsilon$, then it follows from
 \eqref{F26} that $\tilde{\varphi}(\zeta(8\varepsilon,(u,\tau)))<a-\varepsilon$ and (ii) is satisfied. If
 \begin{equation}\label{F30}
   a-\varepsilon\le \tilde{\varphi}(\zeta(t,(u,\tau)))\le a+\varepsilon,\ \ \forall \ t\in [0,8\varepsilon],
 \end{equation}
 then \eqref{F22} and \eqref{F24} imply that $\zeta(t,(u,\tau))\in (M\times\R)\cap S_{\delta}$ for $t\in [0,8\varepsilon]$,
 and so $\zeta(t,(u,\tau))\in B$ for $t\in [0,8\varepsilon]$, it follows from  \eqref{F26}  and \eqref{F30} that
 \begin{eqnarray*}\label{D36}
   \tilde{\varphi}(\zeta(8\varepsilon,(u,\tau)))
   &  =  & \tilde{\varphi}(u,\tau)
           +\int_{0}^{8\varepsilon}\f{\mathrm{d}}{\mathrm{d}t}\tilde{\varphi}(\zeta(t,(u,\tau)))\mathrm{d}t\nonumber\\
   & \le & \tilde{\varphi}(u,\tau)-\f{1}{4}\int_{0}^{8\varepsilon}\varrho(\zeta(t,(u,\tau)))\mathrm{d}t\nonumber\\
   & \le & a+\varepsilon-2\varepsilon=c-\varepsilon,
 \end{eqnarray*}
 and (ii) is also satisfied.
 \end{proof}

 \begin{theorem}\label{thm 2.14}
  Assume that $\tilde{\theta}\in \R$, $\tilde{\varphi}\in \mathcal{C}^1(E\times \R,\R)$ and $\tilde{\Upsilon}
 \subset M\times\R$ is a closed set. Let
 \begin{equation}\label{Gat0}
   \tilde{\Gamma}:=\left\{\tilde{\gamma}\in \mathcal{C}([0,1], M\times\R):  \tilde{\gamma}(0)\in \tilde{\Upsilon}, \ \tilde{\varphi}(\tilde{\gamma}(1)) < \tilde{\theta}\right\}.
 \end{equation}
 If $\tilde{\varphi}$ satisfies
 \begin{equation}\label{F40}
   \tilde{a}:=\inf_{\tilde{\gamma}\in \tilde{\Gamma}}\max_{t\in [0, 1]}\tilde{\varphi}(\tilde{\gamma}(t))
     >\tilde{b}:=\sup_{\tilde{\gamma}\in \tilde{\Gamma}}\max\left\{\tilde{\varphi}(\tilde{\gamma}(0)), \tilde{\varphi}(\tilde{\gamma}(1))\right\},
 \end{equation}
 then, for every $\varepsilon\in (0,(\tilde{a}-\tilde{b})/2)$, $\delta>0$ and $\tilde{\gamma}_*\in \tilde{\Gamma}$
 such that
 \begin{equation}\label{F42}
   \sup_{t\in [0, 1]}\tilde{\varphi}(\tilde{\gamma}_*(t))\le \tilde{a}+\varepsilon,
 \end{equation}
 there exists $(v,\tau)\in M\times\R$ such that
 \begin{enumerate}[{\rm (i)}]
  \item $\tilde{a}-2\varepsilon\le \tilde{\varphi}(v,\tau)\le \tilde{a}+2\varepsilon$;
  \item $\min_{t\in [0,1]}\|(v,\tau)-\tilde{\gamma}_*(t)\|_{E\times\R}\le 2\delta$;
  \item $\left\|\tilde{\varphi}|_{M\times\R}'(v,\tau)\right\|\le \f{8\varepsilon}{\delta}$.
 \end{enumerate}
 \end{theorem}

 \begin{proof}
 Since $\varepsilon\in (0,(\tilde{a}-\tilde{b})/2)$, then one has
 \begin{equation}\label{F45}
   \tilde{a}-2\varepsilon>\tilde{b}.
 \end{equation}
  Suppose the conclusions of theorem are false. We apply Lemma \ref{lem 2.13} with $u=v$ and $S=\tilde{\gamma}_*([0,1])$.  Then, there exists $\eta\in \mathcal{C}([0,1]\times (M\times \R), M\times \R)$ such that (i)-(vi) in Lemma \ref{lem 2.13} hold.  We define $\tilde{\gamma}_0(t)=\eta(1,\tilde{\gamma}_*(t))$. Then it follows from \eqref{F40}, \eqref{F45} and (i) in Lemma \ref{lem 2.13} that
 $$
   \tilde{\gamma}_0(0)=\eta(1,\tilde{\gamma}_*(0))=\tilde{\gamma}_*(0),\ \ \tilde{\gamma}_0(1)=\eta(1,\tilde{\gamma}_*(1))=\tilde{\gamma}_*(1),
 $$
 so $\tilde{\gamma}_0\in \tilde{\Gamma}$. By \eqref{F42} and (ii) in Lemma \ref{lem 2.13}, one has
 $$
   \sup_{t\in [0,1]}\tilde{\varphi}(\gamma_0(t))
    =\sup_{t\in [0,1]}\tilde{\varphi}(\eta(1,\tilde{\gamma}_*(t)))\le \tilde{a}-\varepsilon,
 $$
 contradicting the definition of $\tilde{a}$.
 \end{proof}

 As a consequence of Theorem \ref{thm 2.14}, we have:

 \begin{corollary}\label{cor 2.15}
  Assume that $\tilde{\theta}\in \R$, $\tilde{\Upsilon}\subset M\times\R$ is closed set and $\tilde{\varphi}\in \mathcal{C}^1(E\times\R,\R)$
  satisfies \eqref{Gat0} and \eqref{F40}. Let $\{\tilde{\gamma}_n\}\subset
  \tilde{\Gamma}$ be such that
 \begin{equation}\label{F47}
   \sup_{t\in [0, 1]}\tilde{\varphi}(\tilde{\gamma}_n(t))\le \tilde{a}+\f{1}{n},\ \ \forall \ n\in \N.
 \end{equation}
 Then there exists a sequence $\{(v_n,\tau_n)\}\subset M\times\R$ satisfying
 \begin{enumerate}[{\rm (i)}]
  \item $\tilde{a}-\f{2}{n}\le \tilde{\varphi}(v_n,\tau_n)\le \tilde{a}+\f{2}{n}$;
  \item $\min_{t\in [0,1]}\|(v_n,\tau_n)-\tilde{\gamma}_n(t)\|_{E\times\R}\le \f{2}{\sqrt{n}}$;
  \item $\left\|\tilde{\varphi}|_{M\times\R}'(v_n,\tau_n)\right\|\le \f{8}{\sqrt{n}}$.
 \end{enumerate}
 \end{corollary}

{\section{Normalized solutions of $L^2$-supercritical problem \eqref{Pa}}}
 \setcounter{equation}{0}

 In this section, we study $L^2$-supercritical problem \eqref{Pa} and prove Theorem \ref{thm1.1} by making use of Corollary \ref{cor 2.15}.

 \par
   To apply Corollary \ref{cor 2.15}, we let $E=H^1(\R^N)$ (or $H_{\mathrm{rad}}^1(\R^N)$) and $H=L^2(\R^N)$.
 Define the norms of $E$ and $H$ by
 \begin{equation}\label{H12}
   \|u\|_E:= \left[\int_{\R^N}\left(|\nabla u|^{2}+u^2\right)\mathrm{d}x\right]^{1/2},\ \ \|u\|_H:=\f{1}{\sqrt{c}}\left(\int_{\R^N}u^2\mathrm{d}x\right)^{1/2},\ \ \forall \ u\in E.
 \end{equation}
  After identifying $H$ with its dual, we have $E\hookrightarrow H \hookrightarrow E^*$ with continuous injections.
 Set
 \begin{equation}\label{H14}
   M:= \left\{u\in E: \|u\|_2^2=\int_{\R^N}u^2\mathrm{d}x=c\right\}.
 \end{equation}

 \par
 Obviously, under (F0), $\Phi\in \mathcal{C}^1(E,\R)$, and
 \begin{equation}\label{Phd}
   \langle\Phi'(u),u\rangle= \|\nabla u\|_2^2-\int_{\R^N}f(u)u\mathrm{d}x.
 \end{equation}
 Let us define a continuous map $ \beta: H^1(\R^N)\times\R \to H^1(\R^N)$ by
 \begin{equation}\label{Be}
    \beta(v, t)(x):=e^{Nt/2}v(e^tx)\ \ \mbox{for}\  v\in H^1(\R^N),\ \ \ \ \forall\ t\in\R,\ x\in\R^N,
 \end{equation}
 and consider the following auxiliary functional:
 \begin{equation}\label{Tbe}
   \tilde{\Phi}(v,t):=\Phi(\beta(v,t))
  =\f{e^{2t}}{2}\|\nabla v\|_2^2
              -\f{1}{e^{Nt}}\int_{\R^N}F(e^{Nt/2}v)\mathrm{d}x.
 \end{equation}
 We see that $\tilde{\Phi}'$ is of class $\mathcal{C}^1$, and for any $(w,s)\in H^1(\R^N)\times\R$,
 \begin{align}\label{DIv}
   \left\langle\tilde{\Phi}'(v,t),(w,s)\right\rangle
  & = \left\langle\tilde{\Phi}'(v,t),(w,0)\right\rangle+\left\langle\tilde{\Phi}'(v,t),(0,s)\right\rangle\nonumber\\
  & = e^{2t}\int_{\R^N}\nabla v\cdot\nabla w\mathrm{d}x+e^{2t}s\|\nabla v\|_2^2
         -\f{1}{e^{Nt}}\int_{\R^N}f(e^{Nt/2}v)e^{Nt/2}w\mathrm{d}x\nonumber\\
  &\ \   +\f{Ns}{2e^{Nt}}\int_{\R^N}\left[2F(e^{Nt/2}v)-f(e^{Nt/2}v)e^{Nt/2}v\right]\mathrm{d}x\nonumber\\
  & = \left\langle\Phi'(\beta(v,t)),\beta(w,t)\right\rangle+s\mathcal{P}(\beta(v,t)).
 \end{align}
 Let
 \begin{equation}\label{uph}
   u(x):=\beta(v,t)(x)=e^{Nt/2}v(e^tx),\ \ \phi(x):=\beta(w,t)(x)=e^{Nt/2}w(e^tx).
 \end{equation}
 Then
 \begin{equation}\label{H20}
   (u,\phi)_H=\f{1}{c}\int_{\R^N}u(x)\phi(x)\mathrm{d}x=\f{1}{c}\int_{\R^N}v(x)w(x)\mathrm{d}x
    =(v,w)_H.
 \end{equation}
 This shows that
 \begin{equation}\label{H21}
   \phi\in T_u(\mathcal{S}_c)\ \Leftrightarrow \ (w,s)\in \tilde{T}_{(v,t)}(\mathcal{S}_c\times\R),
    \ \ \forall \ t,s\in \R.
 \end{equation}
 It follows from \eqref{DIv}, \eqref{uph} and \eqref{H21} that
 \begin{equation}\label{H24}
   |\mathcal{P}(u)|=\left|\left\langle\tilde{\Phi}'(v,t),(0,1)\right\rangle\right|
    \le \left\|\tilde{\Phi}|_{\mathcal{S}_c\times\R}'(v,t)\right\|
 \end{equation}
 and
 \begin{eqnarray}\label{H25}
   \left\|\Phi|_{\mathcal{S}_c}'(u)\right\|
    &  =  & \sup_{\phi\in T_u(\mathcal{S}_c)}\f{1}{\|\phi\|_E}
              \left|\left\langle\Phi'(u),\phi\right\rangle\right|\nonumber\\
    &  =  & \sup_{\phi\in T_u(\mathcal{S}_c)}\f{1}{\sqrt{\|\nabla \phi\|_2^2+\|\phi\|_2^2}}
             \left|\left\langle\Phi'(\beta(v,t)),\beta(w,t)\right\rangle\right|\nonumber\\
    &  =  & \sup_{\phi\in T_u(\mathcal{S}_c)}\f{1}{\sqrt{\|\nabla \phi\|_2^2+\|\phi\|_2^2}}
             \left|\left\langle\tilde{\Phi}'(v,t),(w,0)\right\rangle\right|\nonumber\\
    & \le & \sup_{(w,0)\in \tilde{T}_{(v,t)}(\mathcal{S}_c\times\R)}\f{e^{|t|}}{\|(w,0)\|_{E\times\R}}
             \left|\left\langle\tilde{\Phi}'(v,t),(w,0)\right\rangle\right|\nonumber\\
    & \le & e^{|t|}\left\|\tilde{\Phi}|_{\mathcal{S}_c\times\R}'(v,t)\right\|.
 \end{eqnarray}

 \begin{lemma}\label{lem 3.1}
 Assume that {\rm (F1)} and {\rm(F2)} hold. Then
  \begin{enumerate}
  \item[{\rm(i)}] there exists $\rho(c)>0$ small enough such that $ \Phi(u)> 0$ if $u\in A_{2\rho}$ and
 \begin{equation}\label{Ak}
  0<\sup_{u\in A_\rho}\Phi(u)<\kappa_c:=\inf\left\{\Phi(u): u\in \mathcal{S}_c, \|\nabla u\|_2^2= 2\rho(c) \right\},
 \end{equation}
 where
 \begin{equation}\label{Ak1}
  A_{\rho}=\left\{u\in \mathcal{S}_c: \|\nabla u\|_2^2\le \rho(c)\right\} \ \ \hbox{and} \ \
  A_{2\rho}=\left\{u\in \mathcal{S}_c: \|\nabla u\|_2^2\le 2\rho(c)\right\};
 \end{equation}
  \item[{\rm(ii)}] there holds:
 \begin{eqnarray}\label{K5+}
    M(c):= \inf_{\gamma\in \Gamma_c}\max_{t\in [0, 1]}\Phi(\gamma(t))\ge \kappa_c> \sup_{\gamma\in \Gamma_c}\max\left\{\Phi(\gamma(0)), \Phi(\gamma(1))\right\},
 \end{eqnarray}
 where
 \begin{equation}\label{Ga3}
   \Gamma_c:=\left\{\gamma\in \mathcal{C}([0,1], \mathcal{S}_c): \|\nabla\gamma(0)\|_2^2\le \rho(c), \Phi(\gamma(1))<0\right\}.
 \end{equation}
 \end{enumerate}
 \end{lemma}

 \par
  The proof of this lemma is standard, so we omit it.

 \begin{lemma}\label{lem 3.2}
 Assume that {\rm (F0)} holds, and that there exist $\rho(c)>0$ and $\kappa_c>0$ such that \eqref{K5+}
 and \eqref{Ga3} hold. Then there exists a sequence $\{u_n\}\subset \mathcal{S}_c$ such that
 \begin{equation}\label{PCe}
    \Phi(u_n)\rightarrow M(c)>0, \ \   \Phi|_{\mathcal{S}_c}'(u_n) \rightarrow 0\ \
   \mbox{and}\ \ \mathcal{P}(u_n)\rightarrow 0.
 \end{equation}
 \end{lemma}

 \begin{proof}
 Let
 \begin{equation}\label{Ga4}
   \tilde{\Gamma}_c:=\left\{\tilde{\gamma}\in \mathcal{C}([0,1], \mathcal{S}_c\times \R):
    \tilde{\gamma}(0)=(\tilde{\gamma}_1(0),0), \|\nabla \tilde{\gamma}_1(0)\|_2^2 \le \rho(c), \tilde{\Phi}(\tilde{\gamma}(1))<0\right\}
 \end{equation}
 and
 \begin{equation}\label{K6+}
   \tilde{M}(c):=\inf_{\tilde{\gamma}\in \tilde{\Gamma}_c}\max_{t\in [0, 1]}\tilde{\Phi}(\tilde{\gamma}(t)).
 \end{equation}
 For any $\tilde{\gamma}\in \tilde{\Gamma}_c$, it is easy to see that $\gamma=\beta\circ \tilde{\gamma}\in \Gamma_c$
 defined by \eqref{Ga3}. Then it follows from\eqref{Tbe} and \eqref{K5+} that
 \begin{eqnarray}\label{K7+}
   \inf_{\tilde{\gamma}\in \tilde{\Gamma}_c}\max_{t\in [0, 1]}\tilde{\Phi}(\tilde{\gamma}(t))
    &  =  & \inf_{\tilde{\gamma}\in \tilde{\Gamma}_c}\max_{t\in [0, 1]}\Phi(\beta(\tilde{\gamma}(t)))\nonumber\\
    & \ge & \inf_{\gamma\in \Gamma_c}\max_{t\in [0, 1]}\Phi(\gamma(t))\ge \kappa_c\nonumber\\
    &  >  & \sup_{\gamma\in \Gamma_c}\max\left\{\Phi(\gamma(0)), \Phi(\gamma(1))\right\}\nonumber\\
    & \ge & \sup_{\tilde{\gamma}\in \tilde{\Gamma}_c}\max\left\{\tilde{\Phi}(\tilde{\gamma}(0)), \tilde{\Phi}(\tilde{\gamma}(1))\right\}.
 \end{eqnarray}
 This shows that $\tilde{M}(c)\ge M(c)$ and \eqref{F40} holds.

 \par
   On the other hand, for any $\gamma\in \Gamma_c$, let $\tilde{\gamma}(t):=(\gamma(t),0)$. It is easy to
 verify that $\tilde{\gamma}\in \tilde{\Gamma}_c$ and $\Phi(\gamma(t))=\tilde{\Phi}(\tilde{\gamma}(t))$,
 and so, we trivially have $\tilde{M}(c)\le M(c)$. Thus $\tilde{M}(c) = M(c)$.

 \par
   For any $n\in \N$, \eqref{K5+} implies that there exists $\gamma_n\in \Gamma_c$ such that
 \begin{equation}\label{K8+}
   \max_{t\in [0,1]}\Phi(\gamma_n(t)) \le M(c)+\f{1}{n}.
 \end{equation}
 Set $\tilde{\gamma}_n(t):=(\gamma_n(t),0)$. Then applying Corollary \ref{cor 2.15} to $\tilde{\Phi}$, there exists a sequence $\{(v_n,t_n)\}\subset \mathcal{S}_c\times\R$ satisfying
 \begin{enumerate}[{\rm (i)}]
  \item $M(c)-\f{2}{n}\le \tilde{\Phi}(v_n,t_n)\le M(c)+\f{2}{n}$;
  \item $\min_{t\in [0,1]}\|(v_n,t_n)-(\gamma_n(t),0)\|_{E\times\R}\le \f{2}{\sqrt{n}}$;
  \item $\left\|\tilde{\Phi}|_{\mathcal{S}_c\times\R}'(v_n,t_n)\right\|\le \f{8}{\sqrt{n}}$.
 \end{enumerate}
 Let $u_n=\beta(v_n,t_n)$. It follows from \eqref{H24}, \eqref{H25} and (i)-(iii) that \eqref{PCe} holds.

 \end{proof}

 \par
   Similarly, we can obtain the following lemma.

 \begin{lemma}\label{lem 3.2+}
 Assume that {\rm (F0)} holds, and that there exist $\rho(c)>0$ and $\kappa_c>0$ such that \eqref{K5+}
 and \eqref{Ga3} hold. Then there exists a sequence $\{u_n\}\subset \mathcal{S}_c\cap H^1_{\mathrm{rad}}(\R^N)$ such that
 \begin{equation}\label{PCr}
    \Phi(u_n)\rightarrow M(c)>0, \ \   \Phi|_{\mathcal{S}_c}'(u_n) \rightarrow 0\ \
   \mbox{and}\ \ \mathcal{P}(u_n)\rightarrow 0.
 \end{equation}
 \end{lemma}

\begin{lemma}\label{lem 3.3}{\rm\cite{Je-NA}}
 Assume that {\rm (F0)} holds. If there exist $u\in H^1(\R^N)$ and $\lambda\in \R$ such that
 \begin{equation}\label{PB1}
   -\Delta u-\lambda u=f(u),\ \ x\in \R^N,
 \end{equation}
 then $\mathcal{P}(u)=0$, where $\mathcal{P}$ is defined by \eqref{Ju}.
 \end{lemma}

 \begin{proof}[Proof of Theorem {\rm\ref{thm1.1}}] Let $E=H^1_{\mathrm{rad}}(\R^N)$.
  By Lemma \ref{lem 3.2+}, there exists a sequence $\{u_n\}\subset \mathcal{S}_c\cap H^1_{\mathrm{rad}}(\R^N)$
  such that
 \begin{equation}\label{K13}
   \|u_n\|_2^2=c,\ \ \Phi(u_n) \rightarrow M(c),\ \ \ \  \Phi|_{\mathcal{S}_c}'(u_n) \rightarrow 0,\ \ \mathcal{P}(u_n)\rightarrow 0.
 \end{equation}
 From \eqref{Ph}, \eqref{Ju} and \eqref{K13}, one has
 \begin{equation}\label{Phn}
  M(c)+o(1)=\f{1}{2}\|\nabla u_n\|_2^2-\int_{\R^N}F(u_n)\mathrm{d}x
 \end{equation}
 and
 \begin{equation}\label{Pun}
   o(1)= \|\nabla u_n\|_2^2-\f{N}{2}\int_{\R^N}[f(u_n)u_n-2F(u_n)]\mathrm{d}x.
 \end{equation}
 Combining \eqref{Phn} with \eqref{Pun}, we get
 \begin{eqnarray}\label{K23}
   M(c)+1
     & \ge & \f{N}{4}\int_{\R^N}\left[f(u_n)u_n-\left(2+\f{4}{N}\right)F(u_n)\right]\mathrm{d}x.
 \end{eqnarray}
 Now, we prove that $\{\|\nabla u_n\|_2\}$ is bounded. Arguing by contradiction, suppose that
 $\|\nabla u_n\|_2\to\infty$. Let $v_n:=\f{u_n}{\|\nabla u_n\|_2}$. Then $\|\nabla v_n\|_2=1$,
 and $\|v_n\|_2\to 0$ due to \eqref{K13}. Set $\kappa'=\f{\kappa}{\kappa-1}$. Then $N-(N-2)\kappa'>0$, it follows
 from \eqref{GN} that
 \begin{equation}\label{GN0}
    \|v_n\|_{2\kappa'}^{2\kappa'}\le C_1\|v_n\|_2^{N-(N-2)\kappa'}\|\nabla v_n\|_2^{N(\kappa'-1)}=o(1).
 \end{equation}
 By (F2) and (F3), there exist $R_0>0$ and $C_2>0$ such that
 \begin{equation}\label{K24}
    \left|\f{f(t)t-2F(t)}{t^2}\right|^{\kappa}\le C_2\left[f(t)t-\left(2+\f{4}{N}\right)F(t)\right],
    \ \ \forall \ |t|\ge R_0.
 \end{equation}
 Set
 \begin{equation*}\label{Ome}
   \Omega_n:=\left\{x\in\R^N: |u_n(x)|\le R_0\right\}.
 \end{equation*}
 Then by (F1) and (F2), we have
 \begin{equation}\label{BF1}
   \int_{\Omega_n}\f{|f(u_n)u_n-2F(u_n)|}{u_n^2}v_n^2\mathrm{d}x
   \le C_3\|v_n\|_2^2=o(1).
 \end{equation}
 Moreover, from \eqref{K23}, \eqref{GN0}, \eqref{K24} and the H\"older inequality, we have
 \begin{eqnarray}\label{BF2}
  &     & \int_{\R^N\setminus\Omega_n}\f{f(u_n)u_n-2F(u_n)}{u_n^2}v_n^2\mathrm{d}x\nonumber\\
  & \le & \left[\int_{\R^N\setminus\Omega_n}\left|\f{f(u_n)u_n-2F(u_n)}{u_n^2}\right|^{\kappa}
           \mathrm{d}x\right]^{1/\kappa}\left(\int_{\R^N\setminus\Omega_n}|v_n|^{2\kappa'}
           \mathrm{d}x\right)^{1/\kappa'}\nonumber\\
  & \le & C_2^{1/\kappa}\left(\int_{\R^N\setminus\Omega_n}\left[f(u_n)u_n-\left(2+\f{4}{N}\right)F(u_n)\right]
           \mathrm{d}x\right)^{1/\kappa}\|v_n\|_{2\kappa'}^2\nonumber\\
  &  =  & o(1).
 \end{eqnarray}
 Thus, it follows from \eqref{Pun}, \eqref{BF1} and \eqref{BF2} that
 \begin{eqnarray*}
 \f{2}{N}+o(1)
   &  =  & \int_{\R^N}\f{f(u_n)u_n-2F(u_n)}{u_n^2}v_n^2\mathrm{d}x\nonumber\\
   &  =  & \int_{\Omega_n}\f{f(u_n)u_n-2F(u_n)}{u_n^2}v_n^2\mathrm{d}x
            +\int_{\R^N\setminus\Omega_n}\f{f(u_n)u_n-2F(u_n)}{u_n^2}v_n^2\mathrm{d}x\nonumber\\
   &  =  & o(1),
 \end{eqnarray*}
 which is a contradiction. Hence, $\{\|u_n\|_E\}$ is bounded.
 By Lemma \ref{lem 2.2}, one has
 \begin{equation}\label{K25}
   \|u_n\|_2^2=c,\ \ \Phi(u_n) \rightarrow M(c),
   \ \ \Phi'(u_n)-\lambda_nu_n\rightarrow 0,
 \end{equation}
 where
 \begin{equation}\label{K26}
   \lambda_n=\f{1}{\|u_n\|_2^2}\langle\Phi'(u_n),u_n\rangle
     =\f{1}{c}\left[\|\nabla u_n\|_2^2-\int_{\R^N}f(u_n)u_n\mathrm{d}x\right].
 \end{equation}
 Since $\{\|u_n\|_E\}$ is bounded, it follows from (F1), (F2) and \eqref{K26} that $\{|\lambda_n|\}$ is
 also bounded. Thus, we may thus assume, passing to a subsequence if necessary, that $\lambda_n\rightarrow \lambda_c$, $u_n\rightharpoonup \bar{u}$ in $E$, $u_n\to \bar{u}$ in $L^s(\R^N)$ for $s\in (2,2^*)$ and $u_n\rightarrow \bar{u}$ a.e. on $\R^N$. By (F1), (F2) and a standard argument, we can deduce
 \begin{equation}\label{K30}
   \lim_{n\to\infty}\int_{\R^N}f(u_n)\phi\mathrm{d}x= \int_{\R^N}f(\bar{u})\phi\mathrm{d}x,
    \ \ \forall \ \phi\in E,
 \end{equation}
 \begin{equation}\label{K31}
   \lim_{n\to\infty}\int_{\R^N}f(u_n)u_n\mathrm{d}x= \int_{\R^N}f(\bar{u})\bar{u}\mathrm{d}x
 \end{equation}
 and
 \begin{equation}\label{K32}
   \lim_{n\to\infty}\int_{\R^N}F(u_n)\mathrm{d}x= \int_{\R^N}F(\bar{u})\mathrm{d}x.
 \end{equation}
 Hence, from \eqref{Phn}, \eqref{Pun}, \eqref{K31} and \eqref{K32}, we deduce
 \begin{eqnarray}\label{K33}
   M(c)
     &  =  & \lim_{n\to\infty}\left\{\f{N}{4}\int_{\R^N}\left[f(u_n)u_n
              -\left(2+\f{4}{N}\right)F(u_n)\right]\mathrm{d}x\right\}\nonumber\\
     &  =  & \f{N}{4}\int_{\R^N}\left[f(\bar{u})\bar{u}
              -\left(2+\f{4}{N}\right)F(\bar{u})\right]\mathrm{d}x.
 \end{eqnarray}
 This, together with (F2), shows that $\bar{u}\ne 0$.
 It follows from \eqref{K25}, \eqref{K30} and the fact that $\lambda_n\rightarrow \lambda_c$ and $u_n\rightharpoonup
 \bar{u}$ in $E$ that
 \begin{equation}\label{K34}
   \Phi'(\bar{u})-\lambda_c\bar{u}=0.
 \end{equation}
 Hence, \eqref{K34} and Lemma \ref{lem 3.3} yield
 \begin{eqnarray}\label{K38}
   \|\nabla \bar{u}\|_2^2-\lambda_c\|\bar{u}\|_2^2
             -\int_{\R^N}f(\bar{u})\bar{u}\mathrm{d}x=0
 \end{eqnarray}
 and
 \begin{eqnarray}\label{K40}
   \|\nabla \bar{u}\|_2^2-\f{N}{2}\int_{\R^N}[f(\bar{u})\bar{u}-2F(\bar{u})]\mathrm{d}x=0.
 \end{eqnarray}
 From (F2), \eqref{K38} and \eqref{K40}, one has
 \begin{eqnarray}\label{K42}
   -\lambda_c\|\bar{u}\|_2^2
             =\int_{\R^N}\left[NF(\bar{u})-\f{N-2}{2}f(\bar{u})\bar{u}\right]\mathrm{d}x>0,
 \end{eqnarray}
 which implies that $\lambda_c<0$. Now from \eqref{Ph}, \eqref{K25}, \eqref{K30}, \eqref{K32} and \eqref{K34}, one has
 \begin{eqnarray}\label{K22}
   0
     &  =  & \lim_{n\to\infty}\left[\langle\Phi'(u_n)-\lambda_nu_n,u_n-\bar{u}\rangle
               -\langle\Phi'(\bar{u})-\lambda_c\bar{u},u_n-\bar{u}\rangle\right]\nonumber\\
     &  =  & \lim_{n\to\infty}\left\{\int_{\R^N}\left(|\nabla (u_n-\bar{u})|^2
              -\lambda_c|u_n-\bar{u}|^2\right)\mathrm{d}x
               -\int_{{\R}^2}[f(u_n)-f(\bar{u})](u_n-\bar{u})\mathrm{d}x\right\}\nonumber\\
     &  =  & \lim_{n\to\infty}\int_{\R^N}\left(|\nabla (u_n-\bar{u})|^2
              -\lambda_c|u_n-\bar{u}|^2\right)\mathrm{d}x.
 \end{eqnarray}
 It follows that $u_n\rightarrow \bar{u}$ in $E$, and so $\|\bar{u}\|_2^2=c$, $\Phi|_{\mathcal{S}_c}'(\bar{u})=0$
 and $\Phi(\bar{u})=M(c)$.

 \par
   Next, we prove that \eqref{Pa} has a ground state solution on $\mathcal{S}_c$. To this end, let
 \begin{equation}\label{Kmd}
   \mathcal{K}_c:=\left\{u\in \mathcal{S}_c\cap H_{\mathrm{rad}}^1 (\R^N): \ \Phi|_{\mathcal{S}_c}'(u)=0\right\},\ \ \bar{m}(c):=\inf_{\mathcal{K}_c}\Phi.
 \end{equation}
 Then $\bar{u}\in \mathcal{K}_c$. It follows from \eqref{Ph}, \eqref{Ju} and (F2) that
 \begin{eqnarray}\label{I35}
    \Phi(u)
    &  =  & \Phi(u)-\f{1}{2}\mathcal{P}(u)\nonumber\\
    &  =  & \f{N}{4}\int_{\R^N}\left[f(u)u-\left(2+\f{4}{N}\right)F(u)\right]\mathrm{d}x\ge 0,
            \ \ \forall \ u\in \mathcal{K}_c.
 \end{eqnarray}
 Thus $\bar{m}(c)\ge 0$. Let $\{u_n\}\subset \mathcal{K}_c$ such that
 \begin{equation}\label{I36}
   \Phi(u_n)\rightarrow \bar{m}(c),\ \ \Phi|_{\mathcal{S}_c}'(u_n)=0.
 \end{equation}
 From \eqref{Ph}, \eqref{Ju}, \eqref{I36} and Lemma \ref{lem 3.3}, one has
 \begin{equation}\label{Phn1}
  \bar{m}(c)+o(1)=\f{1}{2}\|\nabla u_n\|_2^2-\int_{\R^N}F(u_n)\mathrm{d}x
 \end{equation}
 and
 \begin{equation}\label{Pun1}
   0= \|\nabla u_n\|_2^2-\f{N}{2}\int_{\R^N}[f(u_n)u_n-2F(u_n)]\mathrm{d}x.
 \end{equation}
 Combining \eqref{Phn1} with \eqref{Pun1}, we have
 \begin{eqnarray}\label{I38}
   \bar{m}(c)+o(1)
     &  =  & \f{N}{4}\int_{\R^N}\left[f(u_n)u_n
              -\left(2+\f{4}{N}\right)F(u_n)\right]\mathrm{d}x.
 \end{eqnarray}
 From (F1), (F2), \eqref{Sob}, \eqref{GN} and \eqref{Pun1}, we get
 \begin{eqnarray}\label{I39}
   \|\nabla u_n\|_2^2
    &  =  & \f{N}{2}\int_{\R^N}[f(u_n)u_n-2F(u_n)]\mathrm{d}x\nonumber\\
    & \le & \int_{\R^N}\left(\f{1}{2}c^{-2/N}\mathcal{C}_{N,2+4/N}^{-2-4/N}|u_n|^{2+\f{4}{N}}+C_4|u_n|^{2^*}\right)\mathrm{d}x\nonumber\\
    & \le & \f{1}{2}\|\nabla u_n\|_2^2+C_5\|\nabla u_n\|_2^{2^*},
 \end{eqnarray}
 which, together with the fact that $u_n\ne 0$, yields that $\|\nabla u_n\|_2^2\ge C_6$.
 Let $v_n:=\f{u_n}{\|\nabla u_n\|_2}$. Then $\|\nabla v_n\|_2=1$, and $\|v_n\|_2^2\le c/C_6$. Set $\kappa'=\f{\kappa}{\kappa-1}$.
 By \eqref{GN}, one has
 \begin{equation}\label{I42}
    \|v_n\|_{2\kappa'}^{2\kappa'}\le C_7\|v_n\|_2^{N-(N-2)\kappa'}\|\nabla v_n\|_2^{N(\kappa'-1)}
      \le C_8\min\left\{1,\|\nabla u_n\|_2^{-N+(N-2)\kappa'}\right\}.
 \end{equation}
 Hence, it follows from (F3$'$), \eqref{Pun1}, \eqref{I38}, \eqref{I42} and the H\"older inequality that
 \begin{eqnarray*}
  \f{2}{N}
   &  =  & \int_{\R^N}\f{f(u_n)u_n-2F(u_n)}{u_n^2}v_n^2\mathrm{d}x\nonumber\\
   & \le & \left[\int_{\R^N}\left(\f{f(u_n)u_n-2F(u_n)}
           {u_n^2}\right)^{\kappa}\mathrm{d}x\right]^{1/\kappa}\|v_n\|_{2\kappa'}^2\nonumber\\
   & \le & \mathcal{C}_0^{1/\kappa}\left\{\int_{\R^N}\left[Nf(u_n)u_n-(2N+4)F(u_n)\right]
           \mathrm{d}x\right\}^{1/\kappa}\|v_n\|_{2\kappa'}^2\nonumber\\
   & \le & (4\mathcal{C}_0)^{1/\kappa}\left[C_8\min\left\{1,\|\nabla u_n\|_2^{-N+(N-2)\kappa'}\right\}
            \right]^{1/\kappa'}(\bar{m}(c)+o(1))^{1/\kappa}.
 \end{eqnarray*}
 This shows that $\bar{m}(c)>0$ and $\{\|\nabla u_n\|_2\}$ is bounded. By a standard argument, we can prove
 that there exists $\tilde{u}\in \mathcal{K}_c$ such that $\Phi(\tilde{u})=\bar{m}(c)$.
 \end{proof}

{\section{Normalized solutions for mixed problem \eqref{Pa1} with Sobolev  critical exponent}}
 \setcounter{equation}{0}

 \par
  In this section, we let $f(t)=\mu |t|^{q-2}t+|t|^{2^*-2}t, \gamma_q:=N(q-2)/2q$, and study the
  Schr\"odinger equation \eqref{Pa1} with Sobolev  critical exponent and mixed dispersion.

 \vskip4mm
{\subsection{$L^2$-subcritical perturbation}}
 \setcounter{equation}{0}

 \par
   In this subsection, we always assume that $2< q < 2+\f{4}{N} $ and shall prove Theorem \ref{thm 1.2}.

 \par
   Let $N\ge 3$ and let $c_0$ and $\rho_0$ are given in \cite{JJLV,JL-MA}, where $\rho_0$ depending only on
 $c_0 > 0$ but not on $c \in (0, c_0)$. As in \cite{JL-MA}, we define the sets $V(c)$ and $\partial V(c)$
 as follows:
 \begin{equation}\label{Vc}
   V(c):=\{u\in S_c: \|\nabla u\|_2^2<\rho_0\},\ \ \partial V(c):=\{u\in S_c: \|\nabla u\|_2^2=\rho_0\}.
 \end{equation}
 For any $\mu>0$ and any $c\in (0,c_0)$, the set $V(c) \subset  \mathcal{S}_c$ having the property:
 \begin{equation}\label{mmu}
   m_{\mu}(c)=\inf_{u\in V(c)}\Phi_{\mu}(u)<0<\inf_{u\in \partial V(c)}\Phi_{\mu}(u).
 \end{equation}

 \begin{lemma}\label{lem 4.3}{\rm\cite[Proposition 2.1]{JL-MA}}
 Let $N\ge 3$. For any $\mu > 0$ and any $c\in (0,c_0)$, $m_{\mu}(c)$ is reached by a positive, radially symmetric non-increasing function, denoted $u_c\in V(c)\setminus (\partial V(c))$ that satisfies, for a $\lambda_c<0$,
 \begin{equation}\label{Pa2}
     -\Delta u_c-\mu |u_c|^{q-2}u_c-|u_c|^{2^*-2}u_c=\lambda_cu_c,\ \ \ \   x\in \R^N.
 \end{equation}
 \end{lemma}

 \begin{lemma}\label{lem 4.1}
 Let $N\ge 3$. For any $\mu > 0$ and any $c\in (0,c_0)$, there exists $\kappa_{\mu,c}>0$ such that
 \begin{equation}\label{Mu2}
   M_{\mu}(c):=\inf_{\gamma\in \Gamma_{\mu,c}}\max_{t\in [0, 1]}\Phi_{\mu}(\gamma(t))\ge \kappa_{\mu,c}
     >\sup_{\gamma\in \Gamma_{\mu,c}}\max\left\{\Phi_{\mu}(\gamma(0)), \Phi_{\mu}(\gamma(1))\right\},
 \end{equation}
 where
 \begin{equation}\label{Ga5}
   \Gamma_{\mu,c}=\left\{\gamma\in \mathcal{C}([0,1], \mathcal{S}_c\cap H^1_{\mathrm{rad}}(\mathbb{R}^N)): \gamma(0)=u_c, \Phi_{\mu}(\gamma(1))<2m_{\mu}(c)\right\}.
 \end{equation}
 \end{lemma}

 \begin{proof}
 Set $\kappa_{\mu,c}:=\inf_{u\in \partial V(c)}\Phi_{\mu}(u)$. By \eqref{mmu}, $\kappa_{\mu,c}>0$.
 Let $\gamma\in \Gamma_{\mu,c}$ be arbitrary. Since $\gamma(0)=u_c\in V(c)\setminus (\partial V(c))$, and $\Phi_{\mu}(\gamma(1))<2m_{\mu}(c)$, necessarily in view of \eqref{mmu} $\gamma(1)\not\in V(c)$.
 By continuity of $\gamma(t)$ on $[0,1]$, there exists a $t_0 \in (0,1)$ such that $\gamma(t_0)\in
 \partial V(c)$, and so $\max_{t\in [0, 1]}\Phi_{\mu}(\gamma(t))\ge \kappa_{\mu,c}$. Thus, \eqref{Mu2} holds.
 \end{proof}

 \par
   Let $E=H^1_{\mathrm{rad}}(\R^N)$ and set
 \begin{eqnarray}\label{Ga6}
   \tilde{\Gamma}_{\mu,c}
    & :=  & \left\{\tilde{\gamma}\in \mathcal{C}([0,1], (\mathcal{S}_c\cap H^1_{\mathrm{rad}}(\R^N))\times \R):
             \tilde{\gamma}(0)=(u_c,0), \tilde{\Phi}_{\mu}(\tilde{\gamma}(1))<2m_{\mu}(c)\right\}. \ \
 \end{eqnarray}
 Replace $\Gamma_{c}$ and $\tilde{\Gamma}_{c}$ in Lemmas \ref{lem 3.1} and \ref{lem 3.2} by $\Gamma_{\mu,c}$ and $\tilde{\Gamma}_{\mu,c}$, respectively, we can prove the following lemma by the same arguments as in the proof of
 Lemma \ref{lem 3.2}.

 \begin{lemma}\label{lem 4.2}
 Let $N\ge 3$. For any $\mu > 0$ and any $c\in (0,c_0)$, there exists a sequence $\{u_n\}\subset \mathcal{S}_c\cap H^1_{\mathrm{rad}}(\R^N)$ such that
 \begin{equation}\label{PCe2}
    \Phi_{\mu}(u_n)\rightarrow M_{\mu}(c)>0, \ \   \Phi_{\mu}|_{\mathcal{S}_c}'(u_n) \rightarrow 0\ \
   \mbox{and}\ \ \mathcal{P}_{\mu}(u_n)\rightarrow 0.
 \end{equation}
 \end{lemma}

 \par
  Let $A_N:=[N(N-2)]^{(N-2)/4}$. Now we define functions $U_n(x):=\Theta_n(|x|)$, where
 \begin{equation}\label{wn}
   \Theta_n(r)=A_N
   \begin{cases}
     \left(\f{n}{1+n^2r^2}\right)^{\f{N-2}{2}}, \ \ & 0\le r < 1;\\
     \left(\f{n}{1+n^2}\right)^{\f{N-2}{2}}(2-r), \ \ & 1\le r< 2;\\
     0, \ \ & r\ge 2.
   \end{cases}
 \end{equation}
 Computing directly, we have
 \begin{eqnarray}\label{M12}
   \|U_n\|_2^2
    &  =  & \int_{\R^N}|U_n|^2\mathrm{d}x=\omega_{N-1}\int_{0}^{+\infty}r^{N-1}|\Theta_n(r)|^2\mathrm{d}r\nonumber\\
    &  =  & \omega_{N-1}A_N^2\left[\int_{0}^{1}\f{n^{N-2}r^{N-1}}{\left(1+n^2r^2\right)^{N-2}}\mathrm{d}r
             +\left(\f{n}{1+n^2}\right)^{N-2}\int_{1}^{2}r^{N-1}(2-r)^2\mathrm{d}r\right]\nonumber\\
    &  =  & \omega_{N-1}A_N^2\left[\f{1}{n^2}\int_{0}^{n}\f{s^{N-1}}{\left(1+s^2\right)^{N-2}}\mathrm{d}s
             +\f{2^{N+3}-(N^2+5N+8)}{N(N+1)(N+2)}\left(\f{n}{1+n^2}\right)^{N-2}\right]\nonumber\\
    &  =  &  \omega_{N-1}A_N^2\left[\f{\xi(n)}{n^2}
             +\f{2^{N+3}-(N^2+5N+8)}{N(N+1)(N+2)}\left(\f{n}{1+n^2}\right)^{N-2}\right]\nonumber\\
    &  =  &  O\left(\f{\xi(n)}{n^{2}}\right),\ \ n\to\infty,
 \end{eqnarray}
 where $\omega_{N-1}$ is the measure of the unit sphere in $\R^N$,
 \begin{equation}\label{M13}
   \xi(n):=\int_{0}^{n}\f{s^{N-1}}{\left(1+s^2\right)^{N-2}}\mathrm{d}s=
   \begin{cases}
     O(n), \ \ & \mbox{if} \ N=3;\\
     O\left(\log (1+n^2)\right), \ \ & \mbox{if} \ N =4;\\
     O(1), \ \ & \mbox{if} \ N\ge 5,
   \end{cases}
 \end{equation}
 \begin{eqnarray}\label{M14}
   \|\nabla U_n\|_2^2
    &  =  & \int_{\R^N}|\nabla U_n|^2\mathrm{d}x
             =\omega_{N-1}\int_{0}^{+\infty}r^{N-1}|\Theta_n'(r)|^2\mathrm{d}r\nonumber\\
    &  =  & \omega_{N-1}A_N^2\left[(N-2)^2\int_{0}^{1}\f{n^{N+2}r^{N+1}}{\left(1+n^2r^2\right)^{N}}\mathrm{d}r
             +\left(\f{n}{1+n^2}\right)^{N-2}\int_{1}^{2}r^{N-1}\mathrm{d}r\right]\nonumber\\
    &  =  & \omega_{N-1}A_N^2\left[(N-2)^2\int_{0}^{n}\f{s^{N+1}}{\left(1+s^2\right)^{N}}\mathrm{d}s
             +\f{2^{N}-1}{N}\left(\f{n}{1+n^2}\right)^{N-2}\right]\nonumber\\
    &  =  & \mathcal{S}^{N/2}+\omega_{N-1}A_N^2\left[-(N-2)^2\int_{n}^{+\infty}\f{s^{N+1}}
             {\left(1+s^2\right)^{N}}\mathrm{d}s
             +\f{2^{N}-1}{N}\left(\f{n}{1+n^2}\right)^{N-2}\right]\nonumber\\
    &  =  &  \mathcal{S}^{N/2}+O\left(\f{1}{n^{N-2}}\right),\ \ n\to\infty,
 \end{eqnarray}
 \begin{eqnarray}\label{M16}
   \|U_n\|_{2^*}^{2^*}
    &  =  & \int_{\R^N}|U_n|^{2^*}\mathrm{d}x
             =\omega_{N-1}\int_{0}^{+\infty}r^{N-1}|\Theta_n(r)|^{2^*}\mathrm{d}r\nonumber\\
    &  =  & \omega_{N-1}A_N^{2^*}\left[\int_{0}^{1}\f{n^{N}r^{N-1}}{\left(1+n^2r^2\right)^{N}}\mathrm{d}r
             +\left(\f{n}{1+n^2}\right)^{N}
             \int_{1}^{2}r^{N-1}(2-r)^{2^*}\mathrm{d}r\right]\nonumber\\
    &  =  & \omega_{N-1}A_N^{2^*}\left[\int_{0}^{n}\f{s^{N-1}}{\left(1+s^2\right)^{N}}\mathrm{d}s
             +\left(\f{n}{1+n^2}\right)^{N}\int_{0}^{1}s^{2^*}(2-s)^{N-1}\mathrm{d}s\right]\nonumber\\
    &  =  &  \mathcal{S}^{N/2}+O\left(\f{1}{n^N}\right),\ \ n\to\infty
 \end{eqnarray}
 and
 \begin{eqnarray}\label{M20}
   \|U_n\|_q^q
    &  =  & \int_{\R^N}|U_n|^q\mathrm{d}x=\omega_{N-1}\int_{0}^{+\infty}r^{N-1}|\Theta_n(r)|^q\mathrm{d}r\nonumber\\
    & \ge & \omega_{N-1}A_N^q\int_{0}^{1}\f{n^{q(N-2)/2}r^{N-1}}
             {\left(1+n^2r^2\right)^{q(N-2)/2}}\mathrm{d}r\nonumber\\
    &  =  & \f{\omega_{N-1}A_N^q}{n^{[2N-(N-2)q]/2}}\int_{0}^{n}\f{s^{N-1}}
             {\left(1+s^2\right)^{q(N-2)/2}}\mathrm{d}s\nonumber\\
    &  =  &  O\left(\f{1}{n^{[2N-(N-2)q]/2}}\right),\ \ n\to\infty.
 \end{eqnarray}
 Both \eqref{M12} and \eqref{M14} imply that $U_n\in E$.

 \begin{lemma}\label{lem 4.4}
  Let $N\ge 3$, $2<q<2+\f{4}{N}$ and $c\in (0,c_0)$. Then there holds:
 \begin{equation}\label{M24}
   M_{\mu}(c)< m_{\mu}(c)+\f{\mathcal{S}^{N/2}}{N}.
 \end{equation}
 \end{lemma}

 \begin{proof}  Inspired by \cite[Lemma 3.1]{WW-JFA}, let $u_c\in H^1_{\mathrm{rad}}(\mathbb{R}^N)$ be given in Lemma \ref{lem 4.3}. Then by
 Lemmas \ref{lem 3.3} and \ref{lem 4.3}, we have
 \begin{equation}\label{M05}
   \|u_{c}\|_2^2=c,\ \ \Phi_{\mu}(u_{c})=m_{\mu}(c),\ \ \lambda_c\|u_c\|_2^2=-\mu(1-\gamma_q)\|u_c\|_q^q,
  \ \ u_{c}(x)>0, \ \ \forall \ x\in \R^N.
 \end{equation}
 Set $b:=\inf_{|x|\le 1}u_c(x)$ and $B:=\sup_{|x|\le 2}u_c(x)$. Then $0<b\le B<+\infty$. Hence, it follows from \eqref{wn} and \eqref{M12} that
 \begin{eqnarray}\label{M25}
  \int_{\R^N}u_{c}^{q-1}U_n\mathrm{d}x
    & \le & B^{q-1}\int_{|x|\le 2}U_n\mathrm{d}x=O\left(\f{\sqrt{\xi(n)}}{n}\right),  \ \ n\to\infty,
 \end{eqnarray}
 \begin{eqnarray}\label{M25+}
  \int_{\R^N}u_{c}U_n\mathrm{d}x
    & \le & B\int_{|x|\le 2}U_n\mathrm{d}x=O\left(\f{\sqrt{\xi(n)}}{n}\right),  \ \ n\to\infty
 \end{eqnarray}
 and
 \begin{eqnarray}\label{M26}
  \int_{\R^N}u_{c}|U_n|^{2^*-1}\mathrm{d}x
    & \ge & b\omega_{N-1}\int_{0}^{1}r^{N-1}|\Theta_n(r)|^{(N+2)/(N-2)}\mathrm{d}r\nonumber\\
    &  =  & b\omega_{N-1}A_N^{(N+2)/(N-2)}\int_{0}^{1}\f{n^{(N+2)/2}r^{N-1}}
             {\left(1+n^2r^2\right)^{(N+2)/2}}\mathrm{d}r\nonumber\\
    &  =  & b\omega_{N-1}A_N^{(N+2)/(N-2)}n^{-(N-2)/2}\int_{0}^{n}\f{s^{N-1}}
             {\left(1+s^2\right)^{(N+2)/2}}\mathrm{d}s\nonumber\\
    &  =  &  O\left(\f{1}{n^{(N-2)/2}}\right),\ \ n\to\infty.
 \end{eqnarray}
 By \eqref{Pa2}, \eqref{M12} and \eqref{M05}, one has
 \begin{eqnarray}\label{M28}
  \int_{\R^N}\nabla u_{c}\cdot \nabla U_n\mathrm{d}x
   &  =  & \int_{\R^N}\left(\mu u_c^{q-1}
             +u_c^{2^*-1}+\lambda_cu_c\right)U_n\mathrm{d}x
 \end{eqnarray}
 and for any $t>0$,
 \begin{eqnarray}\label{M29}
  \|u_{c}+tU_n\|_2^2
   &  =  & c+t^2\|U_n\|_2^2+2t\int_{\R^N}u_cU_n\mathrm{d}x\nonumber\\
   &  =  & c+2t\int_{\R^N}u_cU_n\mathrm{d}x +t^2\left[O\left(\f{\xi(n)}{n^{2}}\right)\right], \ \ n\to\infty.
 \end{eqnarray}
 Let $\tau:=\|u_{c}+tU_n\|_2/\sqrt{c}$. Then
 \begin{eqnarray}\label{tau}
  \tau^2 = 1+\f{2t}{c}\int_{\R^N}u_cU_n\mathrm{d}x +t^2\left[O\left(\f{\xi(n)}{n^{2}}\right)\right],
          \ \ n\to\infty.
 \end{eqnarray}
 Now, we define
 \begin{eqnarray}\label{Wn1}
  W_{n,t}(x):=\tau^{(N-2)/2}[u_{c}(\tau x)+tU_n(\tau x)].
 \end{eqnarray}
 Then one has
 \begin{eqnarray}\label{Wn2}
   \|\nabla W_{n,t}\|_2^2=\|\nabla(u_{c}+tU_n)\|_2^2,\ \ \|W_{n,t}\|_{2^*}^{2^*}=\|u_{c}+tU_n\|_{2^*}^{2^*}
 \end{eqnarray}
 and
 \begin{eqnarray}\label{Wn3}
   \|W_{n,t}\|_2^2=\tau^{-2}\|u_{c}+tU_n\|_2^2=c,\ \ \|W_{n,t}\|_q^q=\tau^{q\gamma_q-q}\|u_{c}+tU_n\|_q^q.
 \end{eqnarray}
 It is easy to verify the following inequalities:
 \begin{eqnarray}\label{1+t}
   (1+t)^{p}\ge
   \begin{cases}
     1+pt+pt^{p-1}+t^p, \ \ & \mbox{if} \ p\ge 3;\\
     1+pt^{p-1}+t^p, \ \ & \mbox{if} \ p\ge 2.
   \end{cases}
 \end{eqnarray}

 \par
   In what follows, we distinguish two cases.

 \par
   Case 1). $3\le N\le 5$. In this case, we have $2^*\ge 3$. From \eqref{Phu}, \eqref{M13}-\eqref{M16},
 \eqref{M05}-\eqref{M28} and \eqref{tau}-\eqref{1+t}, we have
 \begin{eqnarray}\label{M30}
   &     & \Phi_{\mu}(W_{n,t})\nonumber\\
   &  =  & \f{1}{2}\|\nabla W_{n,t}\|_2^{2}-\f{1}{2^*}\|W_{n,t}\|_{2^*}^{2^*}
             -\f{\mu}{q}\|W_{n,t}\|_q^q\nonumber\\
   &  =  & \f{1}{2}\|\nabla (u_{c}+tU_n)\|_2^{2}-\f{1}{2^*}\|u_{c}+tU_n\|_{2^*}^{2^*}
            -\f{\mu\tau^{q\gamma_q-q}}{q}\|u_{c}+tU_n\|_q^q\nonumber\\
   & \le & \f{1}{2}\|\nabla u_{c}\|_2^{2}-\f{1}{2^*}\|u_{c}\|_{2^*}^{2^*}
            -\f{\mu\tau^{q\gamma_q-q}}{q}\|u_{c}\|_q^q
            +\f{t^2}{2}\|\nabla U_n\|_2^{2}-\f{t^{2^*}}{2^*}\|U_n\|_{2^*}^{2^*}\nonumber\\
   &     & \ \ \ \ +t\int_{\R^N}\nabla u_{c}\cdot \nabla U_n\mathrm{d}x-t\int_{\R^N}u_{c}^{2^*-1}U_n\mathrm{d}x
               -t^{2^*-1}\int_{\R^N}u_{c}U_n^{2^*-1}\mathrm{d}x\nonumber\\
   &     & \ \ \ \ -\mu\tau^{q\gamma_q-q}t\int_{\R^N}u_{c}^{q-1}U_n\mathrm{d}x\nonumber\\
   &  =  & \Phi(u_{c}) +\f{\mu\left(1-\tau^{q\gamma_q-q}\right)}{q}\|u_{c}\|_q^q
            +\f{t^2}{2}\|\nabla U_n\|_2^{2}-\f{t^{2^*}}{2^*}\|U_n\|_{2^*}^{2^*}\nonumber\\
   &     & \ \ \ \ +\mu\left(1-\tau^{q\gamma_q-q}\right)t\int_{\R^N}u_c^{q-1}U_n\mathrm{d}x
             +\lambda_ct\int_{\R^N}u_cU_n\mathrm{d}x-t^{2^*-1}\int_{\R^N}u_{c}U_n^{2^*-1}\mathrm{d}x\nonumber\\
   &  =  & m_{\mu}(c)+\f{\mu\|u_{c}\|_q^q}{q}\left\{1-\left[1+\f{2t}{c}\int_{\R^N}u_cU_n\mathrm{d}x
            +t^2\left(O\left(\f{\xi(n)}{n^{2}}\right)\right)\right]^{-q(1-\gamma_q)/2}\right\}\nonumber\\
   &     & \ \ \ \ +\f{t^2}{2}\|\nabla U_n\|_2^{2}-\f{t^{2^*}}{2^*}\|U_n\|_{2^*}^{2^*}
               +\lambda_ct\int_{\R^N}u_cU_n\mathrm{d}x-t^{2^*-1}\int_{\R^N}u_{c}U_n^{2^*-1}\mathrm{d}x\nonumber\\
   &     & \ \ \ \ +\mu\left\{1-\left[1+\f{2t}{c}\int_{\R^N}u_cU_n\mathrm{d}x
            +t^2\left(O\left(\f{\xi(n)}{n^{2}}\right)\right)\right]^{-q(1-\gamma_q)/2}\right\}
            \left(t\int_{\R^N}u_c^{q-1}U_n\mathrm{d}x\right)\nonumber\\
   & \le & m_{\mu}(c)+t^2\left(O\left(\f{\xi(n)}{n^{2}}\right)\right)
            +\f{t^{2}}{2}\left[\mathcal{S}^{N/2}+O\left(\f{1}{n^{N-2}}\right)\right]
            -\f{t^{2^*}}{2^*}\left[\mathcal{S}^{N/2}+O\left(\f{1}{n^{N}}\right)\right]\nonumber\\
   &     & \ \ \ \ -t^{2^*-1}\left[O\left(\f{1}{n^{(N-2)/2}}\right)\right]
             +\f{\mu q(1-\gamma_q)t^2}{c}\left(\int_{\R^N}u_c^{q-1}U_n\mathrm{d}x\right)
             \left(\int_{\R^N}u_cU_n\mathrm{d}x\right)\nonumber\\
   & \le & m_{\mu}(c)+\mathcal{S}^{N/2}\left(\f{t^2}{2}-\f{t^{2^*}}{2^*}\right)
            +t^2\left(O\left(\f{\xi(n)}{n^{2}}\right)\right)
            +\f{t^2}{2}\left[O\left(\f{1}{n^{N-2}}\right)\right]\nonumber\\
   &     & \ \ \ \ -\f{t^{2^*}}{2^*}\left[O\left(\f{1}{n^{N}}\right)\right]
            -t^{2^*-1}\left[O\left(\f{1}{n^{(N-2)/2}}\right)\right],\ \ \forall \ t>0.
 \end{eqnarray}
 Noting that
 \begin{equation}\label{M32}
   0<\f{t^2}{2}-\f{t^{2^*}}{2^*}\le \f{1}{N},\ \ \forall \ t>0.
 \end{equation}
 Since $3\le N\le 5$, hence, it follows from \eqref{M13}, \eqref{M30} and \eqref{M32} that there exists
 $\bar{n}\in\N$ such that
 \begin{equation}\label{M33}
   \sup_{t>0}\Phi_{\mu}(W_{\bar{n},t})< m_{\mu}(c)+\f{\mathcal{S}^{N/2}}{N}.
 \end{equation}

 \vskip2mm
 Case 2). $N\ge 6$. In this case, we have $2<q<2^*\le 3$. From \eqref{Phu}, \eqref{M13}-\eqref{M20}, \eqref{M05}-\eqref{M25+}, \eqref{M28} and \eqref{tau}-\eqref{1+t}, we have
 \begin{eqnarray}\label{M34}
   &     & \Phi_{\mu}(W_{n,t})\nonumber\\
   &  =  & \f{1}{2}\|\nabla (u_{c}+tU_n)\|_2^{2}-\f{1}{2^*}\|u_{c}+tU_n\|_{2^*}^{2^*}
            -\f{\mu\tau^{q\gamma_q-q}}{q}\|u_{c}+tU_n\|_q^q\nonumber\\
   & \le & \f{1}{2}\|\nabla u_{c}\|_2^{2}-\f{1}{2^*}\|u_{c}\|_{2^*}^{2^*}
            -\f{\mu\tau^{q\gamma_q-q}}{q}\|u_{c}\|_q^q
            +\f{t^2}{2}\|\nabla U_n\|_2^{2}-\f{t^{2^*}}{2^*}\|U_n\|_{2^*}^{2^*}
            -\f{\mu\tau^{q\gamma_q-q}t^q}{q}\|U_n\|_q^q\nonumber\\
   &     & \ \ +t\int_{\R^N}\nabla u_{c}\cdot \nabla U_n\mathrm{d}x
               -t\int_{\R^N}u_{c}^{2^*-1}U_n\mathrm{d}x
               -\mu\tau^{q\gamma_q-q}t\int_{\R^N}u_{c}^{q-1}U_n\mathrm{d}x\nonumber\\
   &  =  & \Phi(u_{c}) +\f{\mu\left(1-\tau^{q\gamma_q-q}\right)}{q}\|u_{c}\|_q^q
            +\f{t^2}{2}\|\nabla U_n\|_2^{2}-\f{t^{2^*}}{2^*}\|U_n\|_{2^*}^{2^*}
            -\f{\mu\tau^{q\gamma_q-q}t^q}{q}\|U_n\|_q^q\nonumber\\
   &     & \ \ +\mu\left(1-\tau^{q\gamma_q-q}\right)t\int_{\R^N}u_c^{q-1}U_n\mathrm{d}x
             +\lambda_ct\int_{\R^N}u_cU_n\mathrm{d}x\nonumber\\
   & \le & m_{\mu}(c)+t^2\left[O\left(\f{1}{n^{2}}\right)\right]
            +\f{t^{2}}{2}\left[\mathcal{S}^{N/2}+O\left(\f{1}{n^{N-2}}\right)\right]
            -\f{t^{2^*}}{2^*}\left[\mathcal{S}^{N/2}+O\left(\f{1}{n^{N}}\right)\right]\nonumber\\
   &     & \ \ -t^q\left[O\left(\f{1}{n^{N-(N-2)q/2}}\right)\right]\nonumber\\
   & \le & m_{\mu}(c)+\mathcal{S}^{N/2}\left(\f{t^2}{2}-\f{t^{2^*}}{2^*}\right)
            +t^2\left[O\left(\f{1}{n^{2}}\right)\right]
            -\f{t^{2^*}}{2^*}\left[O\left(\f{1}{n^{N}}\right)\right]\nonumber\\
   &     & \ \  -t^q\left[O\left(\f{1}{n^{N-(N-2)q/2}}\right)\right],\ \ \forall \ t>0.
 \end{eqnarray}
 Hence, it follows from \eqref{M32}, \eqref{M34} and the fact $N\ge 6$ and $2<q<2+\f{4}{N}$ that there exists
 $\bar{n}\in\N$ such that \eqref{M33} holds.

 \par
   Next, we prove that \eqref{M24} hold.  Let $\bar{n}\in \N$ be given in \eqref{M33}. By \eqref{M29},
 \eqref{Wn1}, \eqref{Wn2} and \eqref{Wn3}, we have
 \begin{eqnarray}\label{Wn4}
  W_{\bar{n},t}(x):=\tau^{(N-2)/2}[u_{c}(\tau x)+tU_{\bar{n}}(\tau x)],\ \ \|W_{\bar{n},t}\|_2^2=c
 \end{eqnarray}
 and
 \begin{eqnarray}\label{Wn5}
   \|\nabla W_{\bar{n},t}\|_2^2
    &  =  & \|\nabla(u_{c}+tU_{\bar{n}})\|_2^2= \|\nabla u_{c}\|_2^2+t^2\|\nabla U_{\bar{n}})\|_2^2
             +2t\int_{\R^N}\nabla u_c\cdot \nabla U_{\bar{n}}\mathrm{d}x,
 \end{eqnarray}
 where
 \begin{eqnarray}\label{tau2}
   \tau^2=\|u_{c}+tU_{\bar{n}}\|_2^2/c
   &  =  & 1+\f{2t}{c}\int_{\R^N}u_cU_{\bar{n}}\mathrm{d}x +\|U_{\bar{n}}\|_2^2t^2.
 \end{eqnarray}
 It follows from \eqref{M30}, \eqref{M34}, \eqref{Wn4} and \eqref{Wn5} that $W_{\bar{n},t}\in \mathcal{S}_c$
 for all $t\ge 0$, $W_{\bar{n},0}=u_c$ and $\Phi_{\mu}(W_{\bar{n},t}) <2m_{\mu}(c)$ for large $t>0$. Thus, there
 exist $\hat{t}>0$ such that
 \begin{eqnarray}\label{tau2}
   \Phi_{\mu}(W_{\bar{n},\hat{t}}) <2m_{\mu}(c).
 \end{eqnarray}
 Let $\gamma_{\bar{n}}(t):=W_{\bar{n},t\hat{t}}$. Then $\gamma_{\bar{n}}\in \Gamma_{\mu,c}\cap H^1_{\mathrm{rad}}(\mathbb{R}^N)$ defined by
 \eqref{Ga5}. Hence, it follows from \eqref{Mu2} and \eqref{M33} that \eqref{M24} holds.
 \end{proof}

 \begin{lemma}\label{lem 4.5}{\rm\cite[Proposition 1.11]{JL-MA}}
  Let $N\ge 3$ and $2<q<2+\f{4}{N}$. For any $c\in (0,c_0)$, if \eqref{M24} holds,
 then the sequence $\{u_n\}$ obtained in Lemma {\rm\ref{lem 4.2}} is, up to subsequence,
 strongly convergent in $H_{\mathrm{rad}}^1(\R^N)$.
 \end{lemma}

 \begin{proof}[Proof of Theorem {\rm\ref{thm 1.2}}] The proof of Theorem \ref{thm 1.2} follows directly
 combining Lemmas \ref{lem 4.2}, \ref{lem 4.4} and \ref{lem 4.5}.
 \end{proof}

{\subsection{$L^2$-critical and $L^2$-supercritical perturbation}}

 \par
   In this subsection, we always assume that $2+\f{4}{N} \le q < 2^*$ and shall prove Theorem \ref{thm 1.3}.

 Let
 \begin{equation}\label{h1h2}
   h(t):=\f{1-t^2}{2}-\f{1-t^{2^*}}{2^*}.
 \end{equation}
 It is easy to see that $h(t)>h(1)=0$ for all $t\in [0,1)\cup (1,+\infty)$.

 \begin{lemma}\label{lem 4.6}
  Let $N\ge 3$, $c>0$, $\mu>0$ and $2+\f{4}{N}\le q < 2^*$. Then there hold
 \begin{equation}\label{PhP}
   \Phi_{\mu}(u)\ge \Phi_{\mu}\left(t^{N/2}u_t\right)+\f{1-t^2}{2}\mathcal{P}_{\mu}(u)+h(t)\|u\|_{2^*}^{2^*},
    \ \ \forall \ u\in \mathcal{S}_c, \ t>0.
 \end{equation}
 \end{lemma}

 \par
   By a simple calculation, we can prove the above lemma. From Lemma \ref{lem 4.6}, we have the following
 corollary.

 \begin{corollary}\label{cor 4.7}
 Let $N\ge 3$, $c>0$, $\mu>0$ and $2+\f{4}{N}\le q < 2^*$. Then for any $u\in \mathcal{M}_{\mu}(c)$, there holds
 \begin{equation}\label{Pmax}
   \Phi_{\mu}(u) = \max_{t> 0}\Phi_{\mu}\left(t^{N/2}u_t\right).
 \end{equation}
 \end{corollary}

 \begin{lemma}\label{lem 4.8}
 Let $N\ge 3$, $c>0$, $\mu>0$ and $2+\f{4}{N}\le q < 2^*$. Then for any
 $u\in \mathcal{S}_c$, there exists a unique $t_u>0$ such that $t_u^{N/2}u_{t_u}\in \mathcal{M}_{\mu}(c)$.
 \end{lemma}

 \par
   The proof of Lemma \ref{lem 4.8} is standard, so we omit it.

 \vskip4mm
 \par
    From Corollary \ref{cor 4.7} and Lemma \ref{lem 4.8}, we have the following lemma.

 \begin{lemma}\label{lem 4.9}
 Let $N\ge 3$, $c>0$, $\mu>0$ and $2+\f{4}{N}\le q<2^*$. Then
 \begin{equation}\label{hmc}
   \inf_{u\in \mathcal{M}_{\mu}(c)}\Phi_{\mu}(u)
   :=\hat{m}_{\mu}(c)=\inf_{u\in \mathcal{S}_c}\max_{t > 0}\Phi_{\mu}\left(t^{N/2}u_t\right).
 \end{equation}
 \end{lemma}

 \par
   In what follows, we set $\alpha(N,q)=+\infty$ if $2+\f{4}{N}<q<2^*$ and $\alpha(N,\bar{q})=\f{1}{2\gamma_{\bar{q}}c^{2/N}\mathcal{C}_{N,\bar{q}}^{\bar{q}}}$.

 \begin{lemma}\label{lem 4.10}
 Let $N\ge 3$, $c>0$, $0<\mu<\alpha(N,q)$ and $2+\f{4}{N}\le q<2^*$. Then
 \begin{enumerate}[{\rm(i)}]
  \item there exists $\vartheta_c>0$ such that $\Phi_{\mu}(u)>0$ and $\mathcal{P}_{\mu}(u)>0$
  if $u\in A_{2\vartheta_c}$, and
 \begin{equation}\label{Ak}
  0<\sup_{u\in A_{\vartheta_c}}\Phi_{\mu}(u)<\inf\left\{\Phi_{\mu}(u): u\in \mathcal{S}_c,
   \ \|\nabla u\|_2^2= 2\vartheta_c \right\},
 \end{equation}
 where
 \begin{equation}\label{Ak1}
  A_{\vartheta_c}=\left\{u\in \mathcal{S}_c: \|\nabla u\|_2^2\le \vartheta_c\right\} \ \ \hbox{and} \ \
 A_{2\vartheta_c}=\left\{u\in \mathcal{S}_a: \|\nabla u\|_2^2\le 2\vartheta_c\right\};
 \end{equation}

 \item[{\rm(ii)}] $\hat{\Gamma}_{\mu,c}=\{\gamma\in \mathcal{C}([0,1],\mathcal{S}_c):\|\nabla \gamma(0)\|_2^2\le \vartheta_c, \Phi_{\mu}(\gamma(1))<0\}\ne \emptyset$ and
  \begin{align}\label{gG0}
   \hat{M}_{\mu}(c)
    &  :=   \inf_{\gamma\in \hat{\Gamma}_{\mu,c}}\max_{t\in[0,1]}\Phi_{\mu}(\gamma(t))
       \ge  \hat{\kappa}_{\mu,c}:=\inf\left\{\Phi_{\mu}(u): u\in \mathcal{S}_c, \|\nabla u\|_2^2
         = 2\vartheta_c \right\}\nonumber\\
    & >     \max_{\gamma\in \hat{\Gamma}_{\mu,c}}\max\{\Phi_{\mu}(\gamma(0)),\Phi_{\mu}(\gamma(1))\}.
\end{align}
\end{enumerate}
 \end{lemma}

 \begin{proof}  i). We distinguish to two cases.
  \par
  Case 1). $2+\f{4}{N}< q<2^*$. In this case, one has $2<q\gamma_q<2^*$. By \eqref{Phu}, \eqref{Pu}, \eqref{Sob}
  and \eqref{GN}, one has
 \begin{eqnarray*}\label{P13}
   \Phi_{\mu}(u)
     &   =  &  \f{1}{2}\|\nabla u\|_2^2-\f{\mu}{q}\|u\|_q^q-\f{1}{2^*}\|u\|_{2^*}^{2^*}\nonumber\\
     & \ge  &  \|\nabla u\|_2^2\left[\f{1}{2}-\f{\mu}{q}c^{(1-\gamma_q)q/2}\mathcal{C}_{N,q}^q
                \|\nabla u\|_2^{q\gamma_q-2}-\f{1}{2^*}\mathcal{S}^{-\f{2^*}{2}}
                \|\nabla u\|_2^{2^*-2}\right],\ \ \forall \ u\in \mathcal{S}_{c}
 \end{eqnarray*}
 and
 \begin{eqnarray*}\label{P14}
   \mathcal{P}_{\mu}(u)
    &   =  &  \|\nabla u\|_2^{2}-\mu\gamma_q\|u\|_q^q-\|u\|_{2^*}^{2^*}\\
    & \ge  &  \|\nabla u\|_2^2\left[1-\mu\gamma_qc^{(1-\gamma_q)q/2}\mathcal{C}_{N,q}^q\|\nabla u\|_2^{q\gamma_q-2}
               -\mathcal{S}^{-\f{2^*}{2}}\|\nabla u\|_2^{2^*-2}\right],\ \ \forall \ u\in \mathcal{S}_{c}.
 \end{eqnarray*}
 Since $q\gamma_q>2$, the above inequalities show there exists $\vartheta_c>0$ such that i) holds.

 \par
  Case 2). $q=2+\f{4}{N}=\bar{q}<2^*$. In this case $2=q\gamma_q<2^*$. By \eqref{Phu}, \eqref{Pu}, \eqref{Sob}
 and \eqref{GN}, one has
 \begin{eqnarray*}\label{P13}
   \Phi_{\mu}(u)
     &   =  & \f{1}{2}\|\nabla u\|_2^2-\f{\mu}{\bar{q}}\|u\|_{\bar{q}}^{\bar{q}}
               -\f{1}{2^*}\|u\|_{2^*}^{2^*}\nonumber\\
     & \ge  & \f{1}{2}\|\nabla u\|_2^2-\f{\mu}{\bar{q}}c^{2/N}\mathcal{C}_{N,\bar{q}}^{\bar{q}}
              \|\nabla u\|_2^{2}-\f{1}{2^*}\mathcal{S}^{-\f{2^*}{2}}\|\nabla u\|_2^{2^*}\nonumber\\
     & \ge  & \|\nabla u\|_2^2\left[\f{1}{4}-\f{1}{2^*}\mathcal{S}^{-\f{2^*}{2}}\|\nabla u\|_2^{2^*-2}\right],
             \ \ \forall \ u\in \mathcal{S}_{c}
 \end{eqnarray*}
 and
 \begin{eqnarray*}\label{P14}
   \mathcal{P}_{\mu}(u)
     &   =  &  \|\nabla u\|_2^{2}-\mu\gamma_{\bar{q}}\|u\|_{\bar{q}}^{\bar{q}}-\|u\|_{2^*}^{2^*}\\
     & \ge  &  \|\nabla u\|_2^2-\mu\gamma_{\bar{q}} c^{2/N}\mathcal{C}_{N,\bar{q}}^{\bar{q}}
               \|\nabla u\|_2^{2}-\mathcal{S}^{-\f{2^*}{2}}\|\nabla u\|_2^{2^*}\nonumber\\
     & \ge  & \|\nabla u\|_2^2\left[\f{1}{2}-\mathcal{S}^{-\f{2^*}{2}}\|\nabla u\|_2^{2^*-2}\right],
             \ \ \forall \ u\in \mathcal{S}_{c}.
 \end{eqnarray*}
 The above inequalities show there exists $\vartheta_c>0$ such that i) holds also.

 \par
   ii).  For any given $w\in \mathcal{S}_c$, we have $\|t^{N/2}w_t\|_2=\|w\|_2$, and so $t^{N/2}w_t\in \mathcal{S}_c$
 for every $t>0$. Then \eqref{Phu} yields
 \begin{equation}\label{Pt1}
  \Phi_{\mu}\left(t^{N/2}w_t\right)=\f{t^2}{2}\|\nabla w\|_2^{2}-\f{\mu t^{q\gamma_q}}{q}\|w\|_q^q
             -\f{t^{2^*}}{2^*}\|w\|_{2^*}^{2^*} \to -\infty \ \ \hbox{as}\ \ t\to +\infty.
 \end{equation}
 Thus we can deduce that there exist $t_1>0$ small enough and $t_2>0$ large enough such that
 \begin{equation}\label{t12}
   \left\|\nabla \left(t_1^{N/2}w_{t_1}\right)\right\|_2^2=t_1^2\|\nabla w\|_2^2 \le \vartheta_c, \ \
   \hbox{and}\ \ \Phi_{\mu}\left(t_2^{N/2}w_{t_2}\right)<0.
 \end{equation}
 Let $\gamma_0(t):=[t_1+(t_2-t_1)t]^{N/2}w_{t_1+(t_2-t_1)t}$. Then $\gamma_0\in \hat{\Gamma}_{\mu,c}$, and so
 $\hat{\Gamma}_{\mu,c}\ne \emptyset$. Now using the intermediate value theorem,
 for any $\gamma\in\hat{\Gamma}_{\mu,c}$, there exists $t_0\in (0,1)$, depending on $\gamma$, such that
 $\|\nabla \gamma(t_0)\|_2^2=2\vartheta_c$ and
 \begin{equation*}
  \max_{t\in[0,1]}\Phi_{\mu}(\gamma(t))\ge \Phi_{\mu}(\gamma(t_0))\ge
  \inf\left\{\Phi_{\mu}(u): u\in \mathcal{S}_c, \|\nabla u\|_2^2= 2\vartheta_c \right\},
 \end{equation*}
 which, together with the arbitrariness of $\gamma\in\hat{\Gamma}_{\mu,c}$, implies
 \begin{equation}\label{E14}
  \hat{M}_{\mu}(c)=\inf_{\gamma\in \hat{\Gamma}_{\mu,c}}\max_{t\in[0,1]}\Phi_{\mu}(\gamma(t))\ge
  \inf\left\{\Phi_{\mu}(u): u\in \mathcal{S}_c, \|\nabla u\|_2^2= 2\vartheta_c \right\}.
 \end{equation}
 Hence, \eqref{gG0} follows directly from  \eqref{Ak} and \eqref{E14}, and the proof is completed.

 \end{proof}

 \begin{lemma}\label{lem 4.11}
  Let $N\ge 3$, $c>0$, $0<\mu<\alpha(N,q)$ and $2+\f{4}{N}\le q<2^*$. Then
 \begin{equation}\label{mcd}
   \hat{M}_{\mu}(c)=\hat{m}_{\mu}(c).
 \end{equation}
 \end{lemma}

 \begin{proof}  We first prove that $\hat{M}_{\mu}(c)\le \hat{m}_{\mu}(c)$. For any $u\in \mathcal{M}_{\mu}(c)$,
 there exist $t_1>0$ small enough and $t_2>1$ large enough such that
 \begin{equation*}\label{t12}
   \left\|\nabla \left(t_1^{N/2}u_{t_1}\right)\right\|_2^2=t_1^2\|\nabla u\|_2^2 \le \vartheta_c, \ \
   \hbox{and}\ \ \Phi_{\mu}\left(t_2^{N/2}u_{t_2}\right)<0.
 \end{equation*}
 Letting
 $$
  \hat{\gamma}(t):=[(1-t) t_1+tt_2]^{N/2}u_{(1-t) t_1+tt_2},\ \ \forall\ t\in[0,1].
 $$
 Then $\hat{\gamma}\in \hat{\Gamma}_{\mu,c}$. By \eqref{Pmax} and the definition of $\hat{M}_{\mu}(c)$, we have
 $$
  \hat{M}_{\mu}(c)\le \max_{t\in [0,1]}\Phi_{\mu}(\hat{\gamma}(t))=\Phi_{\mu}(u),
 $$
 and so $\hat{M}_{\mu}(c)\le \hat{m}_{\mu}(c)$.

 On the other hand, by \eqref{PhP} with $t\to 0$, we have
 \begin{equation*}
  \mathcal{P}_{\mu}(u)
  \le 2\Phi_{\mu}(u),\ \ \forall\ u\in \mathcal{S}_c.
 \end{equation*}
 which implies
 \begin{equation*}
  \mathcal{P}_{\mu}(\gamma(1)) \le 2\Phi_{\mu}(\gamma(1))<0,\ \ \forall\ \gamma\in \hat{\Gamma}_{\mu,c}.
 \end{equation*}
 Since $\|\gamma(0)\|_2^2\le \vartheta_c$, by (i) of Lemma \ref{lem 4.10}, we have $\mathcal{P}_{\mu}(\gamma(0))>0$.
 Hence, any path in $\hat{\Gamma}_{\mu,c}$ has to cross $\mathcal{M}_{\mu}(c)$. This shows that
 \begin{equation*}
 \max_{t\in [0,1]}\Phi_{\mu}(\gamma(t)) \ge \inf_{u\in \mathcal{M}_{\mu}(c)}\Phi_{\mu}(u)=\hat{m}_{\mu}(c),
 \ \ \forall\ \gamma\in \hat{\Gamma}_{\mu,c},
 \end{equation*}
 and so $\hat{M}_{\mu}(c)\ge \hat{m}_{\mu}(c)$ due to the arbitrariness of $\gamma$.
 Therefore, $\hat{M}_{\mu}(c) = \hat{m}_{\mu}(c)$ for any $c>0$, and the proof is completed.
 \end{proof}

 \par
   Set
 \begin{eqnarray*}\label{Ga7}
   \tilde{\hat{\Gamma}}_{\mu,c}
    & :=  & \left\{\tilde{\gamma}\in \mathcal{C}([0,1], \mathcal{S}_c\times \R):
             \tilde{\gamma}(0)=(\tilde{\gamma}_1(0),0), \|\nabla\tilde{\gamma}_1(0)\|_2^2\le \vartheta_c, \tilde{\Phi}_{\mu}(\tilde{\gamma}(1))<0\right\}.
 \end{eqnarray*}
 Replace $\Gamma_{c}$ and $\tilde{\Gamma}_{c}$ in Lemmas \ref{lem 3.1} and \ref{lem 3.2} by $\hat{\Gamma}_{\mu,c}$
 and $\tilde{\hat{\Gamma}}_{\mu,c}$, respectively, we can prove the following lemma by the same arguments as
 in the proof of Lemma \ref{lem 3.2}.

 \begin{lemma}\label{lem 4.12}
 Let $N\ge 3$, $c>0$, $0<\mu<\alpha(N,q)$ and $2+\f{4}{N}\le q<2^*$. There exists a sequence $\{u_n\}\subset
 \mathcal{S}_c$ such that
 \begin{equation}\label{PCe3}
    \Phi_{\mu}(u_n)\rightarrow \hat{M}_{\mu}(c)>0, \ \   \Phi_{\mu}|_{\mathcal{S}_c}'(u_n)
    \rightarrow 0\ \ \mbox{and}\ \ \mathcal{P}_{\mu}(u_n)\rightarrow 0.
 \end{equation}
 \end{lemma}

 \begin{lemma}\label{lem 4.13}
 Let $N\ge 3$, $c>0$, $0<\mu<\alpha(N,\mu)$ and $2+\f{4}{N}\le q<2^*$. Then the function $c \mapsto \hat{m}_{\mu}(c)$ is nonincreasing on $(0,+\infty)$. In particular, if $\hat{m}_{\mu}(c_1)$ is
 achieved, then $\hat{m}_{\mu}(c_1) > \hat{m}_{\mu}(c_2)$ for any $c_2 > c_1$.
 \end{lemma}

 \begin{proof}
  For any $c_2 > c_1 > 0$, it follows from the definition of $\hat{m}_{\mu}(c_1)$ that there exists
 $\{u_n\} \subset \mathcal{M}_{\mu}(c_1)$ such that
 \begin{equation}\label{P24}
   \Phi_{\mu}(u_n)<\hat{m}_{\mu}(c_1)+\f{1}{n},\ \ \forall \ n\in \N.
 \end{equation}
 Let $\tau:=\sqrt{c_2/c_1}$ and $v_n(x):=\tau^{(2-N)/2}u_n(x/\tau)$. Then $\|\nabla v_n\|_2^2=\|\nabla u_n\|_2^2$,
 $\|v_n\|_{2^*}^{2^*}=\|u_n\|_{2^*}^{2^*}$ and $\|v_n\|_2^2=c_2$. By Lemma \ref{lem 4.8}, there exists $t_n > 0$ such that $t_n^{N/2}(v_n)_{t_n}\in \mathcal{M}_{\mu}(c_2)$. Then it follows from \eqref{Phu}, \eqref{P24} and Corollary
 \ref{cor 4.7} that
 \begin{eqnarray}\label{P25}
   \hat{m}_{\mu}(c_2)
   & \le & \Phi_{\mu}\left(t_n^{N/2}(v_n)_{t_n}\right)\nonumber\\
   &  =  & \f{t_n^2}{2}\|\nabla v_n\|_2^2-\f{t_n^{2^*}}{2^*}\|v_n\|_{2^*}^{2^*}
            -\f{\mu t_n^{q\gamma_q}}{q}\|v_n\|_{q}^{q}\nonumber\\
   &  =  & \f{t_n^2}{2}\|\nabla u_n\|_2^2-\f{t_n^{2^*}}{2^*}\|u_n\|_{2^*}^{2^*}
            -\f{\mu\tau^{N-q(N-2)/2}t_n^{q\gamma_q}}{q}\|u_n\|_{q}^{q}\nonumber\\
   &  <  &  \Phi_{\mu}\left(t_n^{N/2}(u_n)_{t_n}\right)\le \Phi_{\mu}(u_n)<\hat{m}_{\mu}(c_1)+\f{1}{n},
 \end{eqnarray}
 which shows that $\hat{m}_{\mu}(c_2) \le \hat{m}_{\mu}(c_1)$ by letting $n\to \infty$.

 \par
   If  $\hat{m}_{\mu}(c_1)$ is achieved, i.e., there exists $\tilde{u}\in \mathcal{M}_{\mu}(c_1)$ such that $\Phi_{\mu}(\tilde{u})=\hat{m}_{\mu}(c_1)$. By the same argument as in \eqref{P25}, we can obtain that
 $\hat{m}_{\mu}(c_2)<\hat{m}_{\mu}(c_1)$.

 \end{proof}

 \par
    Next, we give a precise estimation for the energy level $\hat{M}(c)$ given  by \eqref{gG0},
 which helps us to restore the compactness in the critical case. Different from the strategy in \cite{So-JFA},
 we shall introduced an alternative choice of testing functions which allows  to treat, in a unified way,
 the case $N = 3$, the case $N \ge 4$, the $L^2$-critical case $q=\bar{q}$ and the $L^2$-supercritical case
 $q>\bar{q}$ for \eqref{Pa}. To this end, for any fixed $c>0$, we choose $R_n> n^{2/3}$ be such that
 \begin{eqnarray}\label{Rn1}
   c
    &  =  & \omega_{N-1}A_N^2\left\{\f{1}{n^2}\int_{0}^{n^{5/3}}\f{s^{N-1}}{\left(1+s^2\right)^{N-2}}\mathrm{d}s
             +\left(\f{n}{1+n^{10/3}}\right)^{N-2}\right.\nonumber\\
    &     & \ \ \left.\times\f{2R_n^{N+2}-[(N+1)(N+2)R_n^2+2N(N+2)R_nn^{2/3}
              -N(N+1)n^{4/3}]n^{2N/3}}{N(N+1)(N+2)(R_n-n^{2/3})^2}\right\}.\nonumber\\
 \end{eqnarray}
 Note that
 \begin{equation}\label{Rn2}
   \int_{0}^{n^{5/3}}\f{s^{2}}{1+s^2}\mathrm{d}s=n^{5/3}-\arctan \left(n^{5/3}\right),
 \end{equation}
 \begin{equation}\label{Rn2+}
   \int_{0}^{n^{5/3}}\f{s^{3}}{\left(1+s^2\right)^2}\mathrm{d}s
     =\f{1}{2}\left[\log \left(1+n^{10/3}\right)-\f{n^{10/3}}{1+n^{10/3}}\right]
 \end{equation}
 and
 \begin{equation}\label{Rn3}
   \int_{0}^{n^{5/3}}\f{s^{N-1}}{\left(1+s^2\right)^{N-2}}\mathrm{d}s
     \le \f{1}{N-4}\left[1-\f{1}{\left(1+n^{10/3}\right)^{(N-4)/2}}\right],\ \ \forall \ N\ge 5.
 \end{equation}
 From \eqref{Rn1}, \eqref{Rn2}, \eqref{Rn2+} and \eqref{Rn3}, one can deduce that
 \begin{equation}\label{Rn4}
   \lim_{n\to\infty}\f{R_n}{n^{7(N-2)/3N}}= \left[\f{N(N+1)(N+2)c}{2\omega_{N-1}A_N^2}\right]^{1/N}.
 \end{equation}
 Now, we define function $\tilde{U}_n(x):=\tilde{\Theta}_n(|x|)$, where
 \begin{equation}\label{wn2}
   \tilde{\Theta}_n(r)=A_N
   \begin{cases}
     \left(\f{n}{1+n^2r^2}\right)^{\f{N-2}{2}}, \ \ & 0\le r < n^{\f{2}{3}};\\
     \left(\f{n}{1+n^{10/3}}\right)^{\f{N-2}{2}}\f{R_n-r}{R_n-n^{2/3}}, \ \ & n^{\f{2}{3}}\le r< R_n;\\
     0, \ \ & r\ge R_n.
   \end{cases}
 \end{equation}
 Computing directly, we have
 \begin{eqnarray}\label{N12}
    &     & \|\tilde{U}_n\|_2^2\nonumber\\
    &  =  & \int_{\R^N}|\tilde{U}_n|^2\mathrm{d}x
             =\omega_{N-1}\int_{0}^{+\infty}r^{N-1}|\tilde{\Theta}_n(r)|^2\mathrm{d}r\nonumber\\
    &  =  & \omega_{N-1}A_N^2\left[\int_{0}^{n^{2/3}}\f{n^{N-2}r^{N-1}}{\left(1+n^2r^2\right)^{N-2}}\mathrm{d}r
             +\left(\f{n}{1+n^{10/3}}\right)^{N-2}
             \int_{n^{2/3}}^{R_n}r^{N-1}\f{(R_n-r)^2}{(R_n-n^{2/3})^2}\mathrm{d}r\right]\nonumber\\
    &  =  & \omega_{N-1}A_N^2\left\{\f{1}{n^2}\int_{0}^{n^{5/3}}\f{s^{N-1}}{\left(1+s^2\right)^{N-2}}\mathrm{d}s
             +\left(\f{n}{1+n^{10/3}}\right)^{N-2}\right.\nonumber\\
    &     & \ \ \left.\times\f{2R_n^{N+2}-[(N+1)(N+2)R_n^2+2N(N+2)R_nn^{2/3}
              -N(N+1)n^{4/3}]n^{2N/3}}{N(N+1)(N+2)(R_n-n^{2/3})^2}\right\}\nonumber\\
    &  =  &  c,
 \end{eqnarray}
 \begin{eqnarray}\label{N14}
    &     & \|\nabla \tilde{U}_n\|_2^2\nonumber\\
    &  =  & \int_{\R^N}|\nabla \tilde{U}_n|^2\mathrm{d}x
             =\omega_{N-1}\int_{0}^{+\infty}r^{N-1}|\tilde{\Theta}_n'(r)|^2\mathrm{d}r\nonumber\\
    &  =  & \omega_{N-1}A_N^2\left[(N-2)^2\int_{0}^{n^{2/3}}\f{n^{N+2}r^{N+1}}{\left(1+n^2r^2\right)^{N}}\mathrm{d}r
             +\left(\f{n}{1+n^{10/3}}\right)^{N-2}
             \int_{n^{2/3}}^{R_n}\f{r^{N-1}}{(R_n-n^{2/3})^2}\mathrm{d}r\right]\nonumber\\
    &  =  & \omega_{N-1}A_N^2\left[(N-2)^2\int_{0}^{n^{5/3}}\f{s^{N+1}}{\left(1+s^2\right)^{N}}\mathrm{d}s
             +\f{R_n^{N}-n^{2N/3}}{N(R_n-n^{2/3})^2}\left(\f{n}{1+n^{10/3}}\right)^{N-2}\right]\nonumber\\
    &  =  & \mathcal{S}^{N/2}+\omega_{N-1}A_N^2\left[-(N-2)^2\int_{n^{5/3}}^{+\infty}\f{s^{N+1}}
             {\left(1+s^2\right)^{N}}\mathrm{d}s
             +\f{R_n^{N}-n^{2N/3}}{N(R_n-n^{2/3})^2}\left(\f{n}{1+n^{10/3}}\right)^{N-2}\right]\nonumber\\
    &  =  &  \mathcal{S}^{N/2}+O\left(\f{1}{n^{14(N-2)/3N}}\right),\ \ n\to\infty,
 \end{eqnarray}
 \begin{eqnarray}\label{N16}
   \|\tilde{U}_n\|_{2^*}^{2^*}
    &  =  & \int_{\R^N}|\tilde{U}_n|^{2^*}\mathrm{d}x
             =\omega_{N-1}\int_{0}^{+\infty}r^{N-1}|\tilde{\Theta}_n(r)|^{2^*}\mathrm{d}r\nonumber\\
    &  =  & \omega_{N-1}A_N^{2^*}\left[\int_{0}^{n^{2/3}}\f{n^{N}r^{N-1}}{\left(1+n^2r^2\right)^{N}}\mathrm{d}r
             +\left(\f{n}{1+n^{10/3}}\right)^{N}
             \int_{n^{2/3}}^{R_n}\f{r^{N-1}(R_n-r)^{2^*}}{(R_n-n^{2/3})^{2^*}}\mathrm{d}r\right]\nonumber\\
    &  =  & \omega_{N-1}A_N^{2^*}\left[\int_{0}^{n^{5/3}}\f{s^{N-1}}{\left(1+s^2\right)^{N}}\mathrm{d}s
             +\left(\f{n}{1+n^{10/3}}\right)^{N}
             \int_{0}^{R_n-n^{2/3}}\f{s^{2^*}(R_n-s)^{N-1}}{(R_n-n^{2/3})^{2^*}}\mathrm{d}s\right]\nonumber\\
    &  =  & \mathcal{S}^{N/2}+\omega_{N-1}A_N^{2^*}\left[-\int_{n^{5/3}}^{+\infty}\f{s^{N-1}}
             {\left(1+s^2\right)^{N}}\mathrm{d}s\right.\nonumber\\
    &     & \ \ \left.+R_n^N\left(\f{n}{1+n^{10/3}}\right)^{N}
             \int_{0}^{1-n^{2/3}/R_n}\f{s^{2^*}(1-s)^{N-1}}{(1-n^{2/3}/R_n)^{2^*}}\mathrm{d}s\right]\nonumber\\
    &  =  &  \mathcal{S}^{N/2}+O\left(\f{1}{n^{14/3}}\right),\ \ n\to\infty
 \end{eqnarray}
 and
 \begin{eqnarray}\label{N20}
   \|\tilde{U}_n\|_q^q
    &  =  & \int_{\R^N}|\tilde{U}_n|^q\mathrm{d}x=\omega_{N-1}\int_{0}^{+\infty}r^{N-1}|\tilde{\Theta}_n(r)|^q\mathrm{d}r\nonumber\\
    & \ge & \omega_{N-1}A_N^q\int_{0}^{n^{2/3}}\f{n^{q(N-2)/2}r^{N-1}}
             {\left(1+n^2r^2\right)^{q(N-2)/2}}\mathrm{d}r\nonumber\\
    &  =  & \omega_{N-1}A_N^qn^{[(q-2)N-2q]/2}\int_{0}^{n^{5/3}}\f{s^{N-1}}
            {\left(1+s^2\right)^{q(N-2)/2}}\mathrm{d}s\nonumber\\
    &  =  & O\left(\f{1}{n^{N-(N-2)q/2}}\right),\ \ n\to\infty.
 \end{eqnarray}
 Both \eqref{N12} and \eqref{N14} imply that $\tilde{U}_n\in \mathcal{S}_c$.

 \begin{lemma}\label{lem 4.14}
  Let $N\ge 3$, $c>0$, $0<\mu<\alpha(N,q)$ and $2+\f{4}{N}\le q<2^*$.  Then there exists $\bar{n}\in\N$ such that
 \begin{equation}\label{mcv}
   \hat{M}_{\mu}(c)\le \sup_{t>0}\Phi\left(t^{N/2}(\tilde{U}_{\bar{n}})_{t}\right)< \f{\mathcal{S}^{N/2}}{N}.
 \end{equation}
 \end{lemma}

 \begin{proof}

 From \eqref{Phu}, \eqref{N14}, \eqref{N16} and \eqref{N20}, we have
 \begin{eqnarray}\label{P36}
  \Phi(t^{N/2}(\tilde{U}_n)_{t})
   &  =  & \f{t^2}{2}\|\nabla \tilde{U}_n\|_2^{2}-\f{\mu t^{(q-2)N/2}}{q}\|\tilde{U}_n\|_q^q
             -\f{t^{2^*}}{2^*}\|\tilde{U}_n\|_{2^*}^{2^*}\nonumber\\
   & \le & \f{t^2}{2}\left[\mathcal{S}^{N/2}+O\left(\f{1}{n^{14(N-2)/3N}}\right)\right]
            -\f{t^{2^*}}{2^*}\left[\mathcal{S}^{N/2}
            +O\left(\f{1}{n^{14/3}}\right)\right]\nonumber\\
   &     & \ \  -t^{(q-2)N/2}\left[O\left(\f{1}{n^{N-(N-2)q/2}}\right)\right]\nonumber\\
   & \le & \mathcal{S}^{N/2}\left(\f{t^2}{2}-\f{t^{2^*}}{2^*}\right)
            +\f{t^2}{2}\left[O\left(\f{1}{n^{14(N-2)/3N}}\right)\right]
            -\f{t^{2^*}}{2^*}\left[O\left(\f{1}{n^{14/3}}\right)\right]\nonumber\\
   &     & \ \  -t^{(q-2)N/2}\left[O\left(\f{1}{n^{N-(N-2)q/2}}\right)\right], \ \ \forall \ t>0.
 \end{eqnarray}
 Hence, it follows from \eqref{M32},   \eqref{gG0}, \eqref{P36} and the fact $2+\f{4}{N}\le q<2^*$ that
 there exists $\bar{n}\in\N$ such that \eqref{mcv} holds.

 \end{proof}

 \begin{proof} [Proof of Theorem {\rm\ref{thm 1.3}}] In view of Lemmas \ref{lem 4.12} and \ref{lem 4.14}, there exists $\{u_n\}\subset \mathcal{S}_c$
 such that
 \begin{equation}\label{P40}
   \|u_n\|_2^2=c, \ \ \Phi_{\mu}(u_n)\rightarrow \hat{M}_{\mu}(c)<\f{1}{N}\mathcal{S}^{N/2},
   \ \ \Phi_{\mu}(u_n)|_{\mathcal{S}_c}'\rightarrow 0, \ \ \mathcal{P}_{\mu}(u_n)\rightarrow 0.
 \end{equation}
 which, together with \eqref{Phu} and \eqref{Pu} that
 \begin{equation}\label{P41}
   \hat{M}_{\mu}(c)+o(1) = \f{1}{2}\|\nabla u_n\|_2^2-\f{1}{2^*}\|u_n\|_{2^*}^{2^*}
      -\f{\mu}{q}\|u_n\|_{q}^{q}
 \end{equation}
 and
 \begin{equation}\label{P42}
   o(1)=\|\nabla u_n\|_2^2 -\|u_n\|_{2^*}^{2^*}-\mu\gamma_q\|u_n\|_{q}^{q}.
 \end{equation}
 It follows from \eqref{P41} and \eqref{P42} that
 \begin{align}\label{P43}
   \hat{M}_{\mu}(c)+o(1)
    &  =    \f{1}{N}\|u_n\|_{2^*}^{2^*}+\mu\left(\f{q\gamma_q-2}{2q}\|u_n\|_{q}^{q}\right).
 \end{align}
 This shows that $\{\|u_n\|_{2^*}\}$ is bounded. From \eqref{P42}, \eqref{P43} and the H\"older inequality,
 one has
 \begin{eqnarray}\label{P45}
   \|\nabla u_n\|_2^2
   &  =  & \|u_n\|_{2^*}^{2^*}+\mu\gamma_q\|u_n\|_{q}^{q}+o(1)\nonumber\\
   & \le & \|u_n\|_{2^*}^{2^*}
            +\mu\gamma_q\|u_n\|_{2}^{2(2^*-q)/(2^*-2)}\|u_n\|_{2^*}^{2^*(q-2)/(2^*-2)}+o(1)\nonumber\\
   &  =  & \|u_n\|_{2^*}^{2^*}+\mu\gamma_qc^{(2^*-q)/(2^*-2)}\|u_n\|_{2^*}^{2^*(q-2)/(2^*-2)}+o(1).
 \end{eqnarray}
 This shows that $\{u_n\}$ is bounded in $H^1(\R^N)$. Let $\delta:=\limsup_{n\to\infty}\sup_{y\in \R^N}\int_{B_1(y)}|u_n|^2\mathrm{d}x$. We show that $\delta >0$. Otherwise, in light of Lions' concentration compactness principle \cite[Lemma 1.21]{WM}, $\|u_n\|_q \rightarrow 0$. From \eqref{P41} and \eqref{P42},
 one can get
 \begin{equation}\label{P53}
   \f{1}{N}\|u_n\|_{2^*}^{2^*} = \hat{M}_{\mu}(c)+o(1),
 \end{equation}
 which, together with \eqref{P42}, yields
 \begin{equation}\label{P54}
   \|\nabla u_n\|_2^2  =  \|u_n\|_{2^*}^{2^*}+\mu\gamma_q\|u_n\|_{q}^{q}+o(1)
     = N\hat{M}_{\mu}(c)+o(1).
 \end{equation}
 Hence, it follows from \eqref{Sob}, \eqref{P53} and \eqref{P54} that
 \begin{eqnarray}\label{P56}
  N\hat{M}_{\mu}(c)+o(1)
   &  =  & \|u_n\|_{2^*}^{2^*} \le \left(\f{\|\nabla u_n\|_2^2}{\mathcal{S}}\right)^{\f{N}{N-2}}
            =\left(\f{N\hat{M}_{\mu}(c)+o(1)}{\mathcal{S}}\right)^{\f{N}{N-2}}\nonumber\\
   &  =  & \left(\f{N\hat{M}_{\mu}(c)}{\mathcal{S}}\right)^{\f{N}{N-2}}+o(1).
 \end{eqnarray}
 Consequently, $\hat{M}_{\mu}(c)\ge \f{1}{N}\mathcal{S}^{N/2}$, which contradicts \eqref{mcv}. Thus $\delta>0$.
 Without loss of generality, we may assume the existence of $y_n\in \R^N$ such that
 $\int_{B_{1}(y_n)}|u_n|^2\mathrm{d}x> \f{\delta}{2}$. Let
 $\hat{u}_n(x)=u_n(x+y_n)$. Then we have
 \begin{equation}\label{P58}
   \|\hat{u}_n\|_2^2=c, \ \ \mathcal{P}_{\mu}(\hat{u}_n)\rightarrow 0, \ \ \Phi_{\mu}(\hat{u}_n)\rightarrow \hat{M}_{\mu}(c), \ \ \int_{B_1(0)}|\hat{u}_n|^2\mathrm{d}x> \f{\delta}{2}.
 \end{equation}
 Therefore, there exists $\hat{u}\in H^1(\R^N)\setminus \{0\}$ such that, passing to a subsequence,
 \begin{equation}\label{P60}
 \left\{
   \begin{array}{ll}
     \hat{u}_n\rightharpoonup \hat{u}, & \mbox{in} \ H^1(\R^N); \\
     \hat{u}_n\rightarrow \hat{u}, & \mbox{in} \ L_{\mathrm{loc}}^s(\R^N), \ \forall \ s\in [1, 2^*);\\
     \hat{u}_n\rightarrow \hat{u}, & \mbox{a.e. on} \ \R^N.
   \end{array}
 \right.
 \end{equation}
 By Lemma \ref{lem 2.2}, one has
 \begin{equation}\label{P65}
    \Phi_{\mu}'(\hat{u}_n)-\lambda_n\hat{u}_n\rightarrow 0,
 \end{equation}
 where
 \begin{equation}\label{P66}
   \lambda_n=\f{1}{\|\hat{u}_n\|_2^2}\langle\Phi_{\mu}'(\hat{u}_n),\hat{u}_n\rangle
     =\f{1}{c}\left[\|\nabla \hat{u}_n\|_2^2-\mu\|\hat{u}_n\|_q^q-\|\hat{u}_n\|_{2^*}^{2^*}\right].
 \end{equation}
 Since $\{ \hat{u}_n\}$ is bounded in $H^1(\mathbb{R}^N)$, it follows from \eqref{P66} that $\{|\lambda_n|\}$ is
 also bounded. Thus, we may thus assume, passing to a subsequence if necessary, that $\lambda_n\rightarrow \lambda_c$. By a standard argument, we can deduce
 \begin{equation}\label{P68}
   \Phi_{\mu}'(\hat{u})-\lambda_c\hat{u}=0.
 \end{equation}
 Hence, by Lemma \ref{lem 3.3}, one has
 \begin{equation}\label{P70}
   \mathcal{P}_{\mu}(\hat{u})=\|\nabla \hat{u}\|_2^2-\mu\gamma_q\|\hat{u}\|_q^q-\|\hat{u}\|_{2^*}^{2^*}=0.
 \end{equation}
 Combining \eqref{P68} with \eqref{P70}, it is easy to deduce that $\lambda_c<0$. Set $\|\hat{u}\|_2^2:=\hat{c}$. Then $0<\hat{c}\le c$, and \eqref{P70} shows that $\hat{u}\in \mathcal{M}_{\mu}(\hat{c})$. From \eqref{Phu}, \eqref{Pu}, \eqref{PhP}, \eqref{P65}, \eqref{P66}, Lemma \ref{lem 4.12} and the weak semicontinuity of norm, one has
 \begin{eqnarray*}
   \hat{M}_{\mu}(c)
   &  =  & \lim_{n\to\infty} \left[\Phi_{\mu}(\hat{u}_n)-\f{1}{2}\mathcal{P}_{\mu}(\hat{u}_n)\right]\nonumber\\
   &  =  & \lim_{n\to\infty}\left[\f{1}{N}\|\hat{u}_n\|_{2^*}^{2^*}
            +\f{\mu}{2q}(q\gamma_q-2)\|\hat{u}_n\|_{q}^{q}\right]\nonumber\\
   & \ge & \f{1}{N}\|\hat{u}\|_{2^*}^{2^*}
            +\f{\mu}{2q}(q\gamma_q-2)\|\hat{u}\|_{q}^{q}\nonumber\\
   &  =  & \Phi_{\mu}(\hat{u})-\f{1}{2}\mathcal{P}_{\mu}(\hat{u})
           = \Phi_{\mu}(\hat{u})\nonumber\\
   & \ge & \hat{m}_{\mu}(\hat{c})\ge \hat{m}_{\mu}(c)=\hat{M}_{\mu}(c),
 \end{eqnarray*}
 which implies
 \begin{equation}\label{P76}
    \|\hat{u}\|_2^2=\hat{c},\ \ \Phi_{\mu}(\hat{u})=\hat{m}_{\mu}(\hat{c})=\hat{m}_{\mu}(c),\ \ \mathcal{P}_{\mu}(\hat{u})=0.
 \end{equation}
 This shows $m(\hat{c})$ is achieved. In view of Lemma \ref{lem 4.13}, $\hat{c}=c$. Thus,
 \begin{equation}\label{P78}
    \lambda_c<0, \ \ \|\hat{u}\|_2^2=c, \ \ \Phi_{\mu}(\hat{u})=\hat{m}_{\mu}(c), \ \
     \mathcal{P}_{\mu}(\hat{u})=0,
 \end{equation}
 Both \eqref{P68} and \eqref{P78} imply the conclusions of Theorem {\rm\ref{thm 1.3}} hold.

 \end{proof}

\section*{Acknowledgments}
The authors would like to express their sincere gratitude to the anonymous referee for his/her careful reading
 and valuable suggestions and comments.
This work is partially supported by the National Natural Science Foundation of China (No: 11971485, No: 12001542),
Hunan Provincial Natural Science Foundation
(No: 2022JJ20048, No: 2021JJ40703).

%\bibliographystyle{elsart-num-sort}
%{model1b-num-names}
%{model1b-num-names}
%{elsarticle-num-names}
%{elsarticle-num}
%{elsarticle-num-names}
%{acm}
%{elsart-num-sort}%{astron}%{unsrt}% {astron}% {amsplain}
%\bibliography{Exponential-bib}

\end{document}